\documentclass[11pt]{amsart}
\usepackage{amsfonts,latexsym,amsthm,amssymb,amsmath,amscd,euscript,tikz, tikz-cd}
\usepackage[alphabetic, msc-links, bibtex-style, nobysame]{amsrefs}
\usepackage{stackengine}
\usepackage{framed}
\usepackage{xfrac}
\usepackage[makeroom]{cancel}
\usepackage{faktor}
\usepackage{braket}
\usepackage{pgf,tikz,pgfplots}
\pgfplotsset{compat=1.17}
\usepgfplotslibrary{fillbetween} 
\usepackage{mathrsfs}
\usetikzlibrary{arrows}
\definecolor{wrwrwr}{rgb}{0.3803921568627451,0.3803921568627451,0.3803921568627451}
\definecolor{rvwvcq}{rgb}{0.08235294117647059,0.396078431372549,0.7529411764705882}
\definecolor{mblue}{rgb}{0.2, 0.3, 0.8}
\definecolor{morange}{rgb}{1, 0.5, 0}
\definecolor{mgreen}{rgb}{0.1, 0.4, 0.2}
\definecolor{mred}{rgb}{0.5, 0, 0}
\definecolor{ForestGreen}{RGB}{34,139,34}
\usepackage{hyperref}
\hypersetup{colorlinks=true,citecolor=ForestGreen,linkcolor = blue,urlcolor =black,linkbordercolor={1 0 0}}
\usepackage{float}

\numberwithin{equation}{section}

\usepackage{lmodern}
\usepackage{supertabular}
\usepackage{verbatim}
\usepackage{amssymb}
\usepackage{enumerate}
\usepackage{stmaryrd}
\usepackage{bbm}
\usepackage{mathtools}
\usepackage{microtype}
\usepackage{setspace}
\usepackage{tikz,tikz-cd}
\usetikzlibrary{matrix,calc,positioning,arrows,decorations.pathreplacing,patterns,knots}

\newtheorem{theorem}{{Theorem}}[section]
\newtheorem*{theorem*}{Theorem}
\newtheorem{lemma}[theorem]{Lemma}
\newtheorem{proposition}[theorem]{Proposition}

\newtheorem{corollary}[theorem]{Corollary}
\newtheorem*{corollary*}{Corollary}

\theoremstyle{definition}

\usepackage{lmodern,url,enumerate,mathtools,microtype}
\usepackage[hmargin = 1in,vmargin=1in]{geometry}
\usepackage{graphicx}
\usepackage{subcaption}

\newcommand{\ve}{\varepsilon}

\newcommand{\mr}[1]{{\rm #1}}
\newcommand{\cA}{\mathcal{A}}

\newcommand{\cL}{\mathcal{L}}

\newcommand{\cP}{\mathcal{P}}

\newcommand{\bR}{\mathbb{R}}
\newcommand{\bS}{\mathbb{S}}\newcommand{\bT}{\mathbb{T}}

\newcommand{\bZ}{\mathbb{Z}}

\newcommand{\nc}{\newcommand}

\nc{\on}{\operatorname}
\nc{\p}{\partial}
\nc{\ol}{\overline}
\nc{\ul}{\underline}
\nc{\pa}{\partial}

\nc{\pb}{\partial_b}
\nc{\pc}{\partial_c}
\nc{\pd}{\partial_d}
\nc{\pe}{\partial_e}
\nc{\pf}{\partial_f}
\nc{\pg}{\partial_g}
\nc{\ph}{\partial_h}
\nc{\pari}{\partial_i}
\nc{\pj}{\partial_j}
\nc{\pk}{\partial_k}
\nc{\pl}{\partial_l}
\nc{\pell}{\partial_\ell}
\nc{\parm}{\partial_m}
\nc{\pn}{\partial_n}
\nc{\po}{\partial_o}
\nc{\pp}{\partial_p}
\nc{\pq}{\partial_q}
\nc{\pr}{\partial_r}
\nc{\ps}{\partial_s}
\nc{\pt}{\partial_t}
\nc{\pu}{\partial_u}
\nc{\pv}{\partial_v}
\nc{\pw}{\partial_w}
\nc{\px}{\partial_x}
\nc{\py}{\partial_y}
\nc{\pz}{\partial_z}

\nc{\Spec}{\on{Spec}}

\nc{\sn}{\mr{sn}}
\nc{\cn}{\mr{cn}}
\nc{\dn}{\mr{dn}}

\numberwithin{equation}{section}
\makeatletter
\@addtoreset{equation}{section}
\makeatother

\title{New minimal surfaces in the sphere from capillary minimal cones}
\date{\today}

\author[Benjy Firester]{Benjy Firester}
\address{Department of Mathematics, MIT \newline 
{\href{mailto:benjyfir@mit.edu}{benjyfir@mit.edu}}}
\author[Raphael Tsiamis]{Raphael Tsiamis}
\address{Department of Mathematics, Columbia University \newline
{\href{mailto:r.tsiamis@columbia.edu}{r.tsiamis@columbia.edu}}}

\begin{document}

\begin{abstract}
    For every $p,q\geq 1$, we construct minimal embeddings of $\mathbb{S}^p \times \mathbb{S}^q \times \mathbb{S}^1$ in $\mathbb{S}^{p + q + 2}$ by doubling the links of free-boundary minimal cones in $\mathbb{R}^{p+q+3}$ with bi-orthogonal symmetry. This solves problems posed by Hsiang-Lawson and Hsiang-Hsiang. The equivariance reduces the minimal surface equation to an ODE, and we prove the existence of capillary minimal cones for every contact angle. We obtain free-boundary solutions as limits of capillary surfaces via a singular shooting problem with infinite initial slope. As the contact angle degenerates to $0$, rescalings of the capillary cones converge to a homogeneous solution of the one-phase Bernoulli problem, further illustrating the connection between one-phase free boundaries and minimal surfaces through the capillary functional.
\end{abstract}

\maketitle
\vspace{-0.3cm}

\section{Introduction}

We construct new minimal surfaces in the sphere as minimal embeddings $\bS^p \times \bS^q \times \bS^1 \hookrightarrow \bS^{p+q+2}$, for any $p,q \geq 1$, by doubling the links of free-boundary minimal cones in $\bR^{n+1}_+$.
These free-boundary surfaces arise as limits of families of capillary minimal surfaces constructed by solving a free boundary ODE.
Under suitable rescalings these cones converge, in the shallow-angle regime, to a homogeneous solution of the one-phase Bernoulli problem.
Unlike the examples in~\cite{FTW-1}, our families of free-boundary cones do not enjoy additional reflection symmetries and exhibit different geometric properties.
Topologically, the links are diffeomorphic to $\bS^{n-k-1}\times \bS^{k-1}\times \bS^1$ with isometry group $O(n-k)\times O(k) \times \bZ_2$, and the induced metric can be expressed as a doubly warped product over the $\bS^1$ factor. 
This construction resolves problems posed by Hsiang-Lawson~\cite{hsiang-lawson-jdg}*{Ch.~III \S~3} and Hsiang-Hsiang~\cite{hsiang-hsiang}*{Problem~3} on the existence of minimal submanifolds in the sphere, and advances the Hsiang-Lawson program in low cohomogeneity through a new, self-contained approach.
Our construction also generalizes results of Carlotto-Schulz and Wang-Wang-Zhou~\cites{carlotto-schulz, spherical-bernstein-xin-zhou} to arbitrary pairs of spherical factors.

The construction of minimal hypersurfaces in spheres $\bS^n$ has received great attention, leading to the development of many key techniques in differential geometry.
These include the Hsiang-Lawson framework for the study of minimal submanifolds with large symmetry groups~\cites{hsiang-lawson-jdg , spherical-bernstein, spherical-bernstein-2 , ferus-karcher , free-boundary-mcgrath }, Kapouleas' gluing and desingularization methods~\cites{doubling-clifford-torus, doubling-two-sphere, kapouleas-mcgrath-cpam , kapouleas-wiygul-annalen , choe-soret }, min-max constructions~\cites{ willmore-conjecture , infinitely-many-MN , marques-neves-song , haslhofer-ketover , XinZhouAnnals, WangZhouFourSpheres, song , ko-min-max , spherical-bernstein-xin-zhou }, and the Allen-Cahn approach~\cites{guaraco-phase-transitions , allen-cahn-3-manifolds}.
Separately, White~\cite{white-mapping-degree} constructed complete minimal surfaces asymptotic to a given minimizing cone $\mathbf{C}$ provided that its link is not a homology sphere. 
The topological complexity of the links of cones is connected to their density, as proved by Ilmanen-White and Bernstein-Wang~\cites{ilmanen-white, bernstein-wang} using techniques from mean curvature flow.

Our first result presents the following:

\begin{theorem}\label{thm:MinimalSurfaceExistence}
For every $p, q \geq 1$, there exists a minimal embedding of $\bS^p \times \bS^q \times \bS^1$ in $\bS^{p+q+2}$. 
\end{theorem}

Theorem~\ref{thm:MinimalSurfaceExistence} is obtained from the study of capillary minimal cones as the other endpoint of a capillary interpolation between free-boundary minimal surfaces and the one-phase problem, first introduced in~\cite{FTW-1}.

\begin{theorem}\label{thm:capillaryConeExistence}
    For all $n \geq 4$, $2 \leq k \leq n-2$, and $\theta \in (0,\tfrac{\pi}{2}]$, there exists a capillary minimal cone $\tilde{\mathbf{C}}_{n,k,\theta}$ that is topologically $C(\bS^{n-k-1}\times \bS^{k-1}\times [0,1])$ with isometry group $O(n-k) \times O(k)$. 
 
    The double of the free-boundary cone $\tilde{\mathbf{C}}_{n,k,\frac{\pi}{2}}$ is a non-isoparametric minimal hypercone $\hat{\mathbf{C}}_{n,k} \subset \bR^{n+1}$.
    Its link $\mathbf{S}_{n,k} \subset \bS^n$ is topologically $\bS^{n-k-1} \times \bS^{k-1} \times \bS^1$, invariant under $O(n-k) \times O(k) \times \bZ_2$.
\end{theorem}
The resulting minimal submanifolds of $\bS^n$ are non-isoparametric, meaning that they have non-constant principal curvatures.
Isoparametric minimal surfaces in the sphere have been classified using algebraic and analytic techniques, cf.~\cites{cecil-chi-jensen, miyaoka, chi-four-curvatures}. 
Theorem~\ref{thm:MinimalSurfaceExistence} is also connected to questions posed by Choe-Fraser~\cite{choe-fraser}*{\S~4} on the possible topology and diffeomorphism types of minimal submanifolds inside manifolds with positive Ricci curvature.

A minimal surface $\Sigma \subset \bS^n$ with reflection symmetry $p_{n+1} \leftrightsquigarrow -p_{n+1}$, restricted to the upper half-space $\bR^{n+1}_+$ and $\bS^n_+$, forms a free-boundary surface. 
Conversely, since the boundary of $\bS^n_+$ is totally geodesic, any free-boundary surface can be doubled to produce an embedded minimal surface in the sphere.
Therefore, when viewed as a graph, any such free-boundary minimal surface must have infinite slope along the boundary equator, making the graphical shooting approach singular in these coordinates.
Using a bi-orthogonal symmetry, the capillary minimal surface equation reduces to a free boundary ODE.
For capillary angles $\theta \in (0,\tfrac{\pi}{2})$, shooting techniques that match the contact angle are well-defined, and we show that the resulting solutions converge uniformly to an infinite-slope shot.
This yields the free-boundary minimal surface in Theorem~\ref{thm:MinimalSurfaceExistence}.

The most symmetric case of Theorem~\ref{thm:MinimalSurfaceExistence} where $n=2k$ (hence $n-k-1=k-1$), recovers the minimal hypertori produced by Carlotto-Schulz~\cite{carlotto-schulz}.
Using numerical methods, Perdomo approximated the spectrum and index of the Carlotto-Schulz examples~\cites{perdomo-spectrum, perdomo-carlotto-schulz}.
Analogues of this construction for constant mean curvature hypersurfaces are discussed in~\cites{huang-wei-2022 , perdomo-navigating}.
Very recently, Wang-Wang-Zhou~\cite{spherical-bernstein-xin-zhou} provided an alternative construction of embedded spheres $\bS^3$ and hypertori $\bS^1 \times \bS^1 \times \bS^1$ inside the $4$-sphere, using equivariant min-max theory.

As a consequence of our general technique, we extend the results of~\cites{carlotto-schulz, spherical-bernstein-xin-zhou} to the capillary setting as follows.
\begin{theorem}\label{thm:carlotto-schulz}
    For every $k \geq 2$, there exists a smooth one-parameter family of capillary cones $\{ \bar{\mathbf{C}}_{k,a} \}_{a \in (0,a^*_k] } \subset \bR^{2k+1}$ that interpolates between a free-boundary minimal cone and a homogeneous solution of the one-phase problem through cones of every angle.
    More precisely,
    \begin{enumerate}[(i)]
        \item As $a \downarrow 0$, the rescaled cones $\frac{1}{a} \bar{\mathbf{C}}_{k,a}$ converge in $C^{\infty}_{\textup{loc}}$ to the graph of a homogeneous solution of the one-phase problem~\eqref{eqn:one-phase-problem}.
        \item As $a \in (0,a^*_k]$, the cones $\bar{\mathbf{C}}_{k,a}$ attain every capillary angle $\theta \in (0,\frac{\pi}{2}]$.
        \item The cone $\bar{\mathbf{C}}_{k,a^*_k}$ is the free-boundary cone $\hat{\mathbf{C}}_{2k,k,\frac{\pi}{2}}$, whose double has link given by a minimal embedded $\bS^{k-1} \times \bS^{k-1} \times \bS^1 \hookrightarrow \bS^{2k}$.
    \end{enumerate}
\end{theorem}
This theorem provides further evidence for the importance of the capillary functional in the study of minimal submanifolds, as it provides a bridge between minimal surfaces and solutions of the one-phase problem~\eqref{eqn:one-phase-problem}.
The first instance of such an interpolating family was discovered in~\cite{FTW-1}.

The general case of $n \neq 2k$ in Theorem~\ref{thm:capillaryConeExistence} (corresponding to $p \neq q$ in Theorem~\ref{thm:MinimalSurfaceExistence}) is more challenging than the case $n=2k$ ($p=q$), which enjoys an additional symmetry.
In the low cohomogeneity framework of Hsiang, Hsiang, and Lawson~\cites{hsiang-lawson-jdg , hsiang-hsiang }, the existence of solutions with certain symmetries was obtained via the study of planar dynamical systems; however, their technique encounters difficulties in the current setting, as discussed in~\cite{carlotto-schulz}.
Carlotto-Schulz obtain their surfaces by studying a non-planar, $3 \times 3$ ODE system.
The shooting method was also used in Angenent's construction of self-shrinking tori~\cite{angenent-tori}.

The main part of our paper is devoted to developing a novel, singular shooting method for graphical solutions of the capillary problem, allowing us to study the minimal surfaces and cones in question via single-variable ODE methods.
Notably, this provides an effective method to approximate these surfaces, which is often inaccessible through pure variational methods.
The existence of a minimally embedded $\bS^{k-1} \times \bS^{k-1} \times \bS^1 \hookrightarrow \bS^{2k}$, connected to a homogeneous one-phase solution by a smooth interpolating family of capillary hypersurfaces, follows using only Lemma~\ref{lemma:change-of-variables} (symmetry under $t \leftrightsquigarrow \sqrt{1-t^2}$) and Lemma~\ref{lemma:bound-at-sqrt} (a uniform bound for shots reaching zero) to prove Theorem~\ref{thm:carlotto-schulz}.
The constructions of Theorems~\ref{thm:MinimalSurfaceExistence} and~\ref{thm:capillaryConeExistence} are also connected to new $O(n-k) \times O(k)$-invariant homogeneous one-phase solutions constructed and studied in a related paper~\cite{six-way}.

\section{Construction of capillary surfaces}\label{section:construction-capillary}
We adopt many conventions from~\cite{FTW-1}.
Let $n \geq 4$ and $2 \leq k \leq n-2$, and decompose $\bR^{n+1} = \bR^{n-k}_x \oplus \bR^{k}_y \oplus \bR_z$.
For the capillary problem, we consider the upper half-space $\bR^{n+1}_+$ with boundary $\Pi := \{z = 0\} = \partial \{z \geq 0 \}$, so we assume $z \geq 0$ throughout.
The $O(n-k)\times O(k)$-action on the $(x,y)$-variables reduces the minimal surface equation to an ODE, where a natural equivariant coordinate system for the analysis is $\rho := \sqrt{|x|^2 + |y|^2}$ and $t := \frac{|y|}{\sqrt{|x|^2 + |y|^2}}$.
As computed in~\cite{FTW-1}, a graphical surface $\mathbf{C} = \{ z = U(|x|,|y|) > 0\}$ is a minimal cone if and only if $U = \rho f(t)$, where $f(t)$ satisfies the ODE
\begin{equation}\label{eqn:ode-star}\tag{$\star$}
    (1-t^2) f'' + (f - tf') + (n-2) \left( 1 + (1-t^2) \frac{(f')^2}{1 + f^2} \right) (f - A_{n,k} f') = 0,
\end{equation}
where $A_{n,k}(t) := t - \frac{k-1}{n-2} t^{-1}$.
The boundary angle for $\mathbf{C}$ at a point $t_*$ where $f(t_*) = 0$ is given by
\begin{equation}\label{eqn:boundaryAngleTerm}
 \theta = \on{arctan} (\sqrt{1 - t_*^2} \, |f'(t_*)|)
\end{equation}
A useful phenomenon in the study of families of capillary solutions is that, as the contact angle $\theta \downarrow 0$, the rescaled functions $\frac{1}{\tan \theta} f_{\theta}$ converge (subsequentially) to a multiple of a solution of the one-phase Bernoulli problem; this idea was utilized in~\cites{improved-regularity,FTW-1}.
For $\lambda \in (0,\infty)$ and $f$ a solution of~\eqref{eqn:ode-star}, the rescaled function $\frac{1}{\sqrt{\lambda}} f$ is a solution of the equation with parameter
\begin{equation}\label{eqn:capillary-ODE} \tag{$\star_\lambda$}
    (1-t^2) f'' + (f - tf') + (n-2) \left( 1 + (1-t^2) \frac{\lambda (f')^2}{1 + \lambda f^2} \right) (f - A_{n,k} f') = 0.
\end{equation}
Motivated by the above limiting behavior, we will study equation~\eqref{eqn:capillary-ODE} uniformly in $\lambda \in [0,\infty]$.
Then, equation~\eqref{eqn:capillary-ODE} includes the one-phase problem in the limit $\lambda \downarrow 0$; the resulting homogeneous solutions are further analyzed in~\cites{FTW-1 , FTW-stability-one-phase , six-way}.
When $\lambda = 0$, the problem~\eqref{eqn:capillary-ODE} specializes to the hypergeometric Legendre operator associated to $p(t) := t^{k-1} (1-t^2)^{\frac{n-k}{2}}$, 
\begin{equation}\label{eqn:legendre-form}
\cL_{n,k} f = (1-t^2) f'' + (n-1) (f - t f') + (k-1) t^{-1} f' = \tfrac{1-t^2}{p(t)} \left[ (p f')' + (n-1) \tfrac{p}{1-t^2} f \right].
\end{equation}
The solutions of $\cL_{n,k} f_0 = 0$ have the significance that $U(x,y) := \rho f_0(t)$ is harmonic in $\bR^n$, while the analogue of the boundary condition~\eqref{eqn:boundaryAngleTerm} implies that $|\nabla U| =1$ along $\partial \{ U > 0 \}$.
Consequently, $U$ is a solution of the \textit{one-phase Bernoulli} problem
\begin{equation}\label{eqn:one-phase-problem}\tag{OP}
        \Delta u = 0 \; \text{ in } \; \{ u > 0 \} 
        \qquad \text{and}\qquad 
        |\nabla u| = 1 \; \text{ on } \; \partial \{ u > 0 \}.
\end{equation}
This problem is closely connected to the theory of minimal surfaces, with many constructions and theorems having direct counterparts.
In low dimensions, classical methods such as the Weierstrass representation and gluing techniques have led to classification results and new examples~\cites{ entire-hairpins, traizet , jerison-kamburov, n-dim-catenoid ,  hines-kolesar-mcgrath}.
We refer the reader to~\cites{ one-phase-simon-solomon , FTW-1 , FTW-stability-one-phase } for some recent results on the one-phase problem and its connections to the theory of minimal surfaces.
As in~\cite{FTW-1}, Theorems~\ref{thm:capillaryConeExistence} and~\ref{thm:carlotto-schulz} deepen these connections by bridging minimal surfaces and one-phase solutions through interpolating families of capillary surfaces.

\subsection{Topology of cones in low cohomogeneity}\label{section:topology}

The coefficients $(1-t^2)$ and $A_{n,k}(t)$ of equation~\eqref{eqn:ode-star} produce singularities at $t=0$ and $t=1$, so any regular solution of this equation~\eqref{eqn:capillary-ODE} either satisfies $f'(0) = 0$ or is defined on a sub-interval $[t_1, t_2] \subset (0,1)$.
The symmetry of the equation under exchanging $k \leftrightsquigarrow n-k$ and $\hat{f}(t) := f(\sqrt{1-t^2})$, which corresponds to relabeling the ambient variables $(x,y) \in \bR^{n-k} \times \bR^k$, ensures that these are the only two possibilities up to ambient isometry, see Lemma~\ref{lemma:change-of-variables}.
In particular, the graph of $U = \rho f(t)$ defines a smooth, complete minimal cone with boundary contained in $\Pi$ if and only if $(i)$ either $f$ is even with $f'(0) = 0$ and $f$ attains a zero at $\pm t_0$ for $0 < t_0 < 1$, or $(ii)$ $f$ has two zeros $0 < t_1 < t_2 < 1$. 
The first case for cones with this symmetry group is treated in~\cite{FTW-1}, and we consider the second case in this paper.
The capillary minimal cone problem for solutions of type $(ii)$ is equivalent to finding a free boundary solution $f$ to equation~\eqref{eqn:ode-star} satisfying $\sqrt{1-t_1^2}|f'(t_1)| = \sqrt{1 - t_2^2}|f'(t_2)| = \theta$.
We will approach this as an ODE shooting problem to find the appropriate $t_1$ for each $\theta$. 

The capillary minimal cone $\mathbf{C} = \text{graph}(\rho f(t))$ admits the spherical parametrization
\[
F(\rho, t , \xi , \eta) = \left( \rho \sqrt{1-t^2}  \xi, \rho t  \eta , \rho f(t)\right) , \qquad \xi \in \bS^{n-k-1}, \; \eta \in \bS^{k-1},
\]
for which $|F(\rho, t , \xi, \eta)|^2 = \rho^2 ( 1+ f(t)^2)$.
Consequently, the link of the capillary cone is a $O(n-k) \times O(k)$-invariant surface $\Sigma \subset \bS^n_+$ inside the upper hemisphere with parametrization
\begin{equation}\label{eqn:link-equation}
    \Sigma := \left\{ \left( \frac{\sqrt{1-t^2}}{\sqrt{1+f^2}} \xi, \frac{t}{\sqrt{1+f^2}} \eta, \frac{f}{\sqrt{1+f^2}} \right) : (\xi, \eta, t) \in \bS^{n-k-1} \times \bS^{k-1}\times I_\theta \right\} \subset \bS^n_+
\end{equation}
where $I_{\theta} \subset [0,1]$ denotes the positive phase of $f$.
For cones of type $(ii)$, the resulting function $U = \rho f(t)$ corresponds to a graph over a conical annulus in the base, of the form
\[
\Gamma := \{ (x,y) : t_1 \rho < |y| < t_2 \rho \}
\]
with two boundary components on $\Pi = \{ p_{n+1} \geq 0 \}$, namely the two cones $\{ |y| = t_i \rho \}$.
Topologically, the link of such a cone is diffeomorphic to $\bS^{n-k-1} \times \bS^{k-1} \times I_\theta$.
We denote such a capillary cone with contact angle $\theta$ by $\tilde{\mathbf{C}}_{n,k,\theta}$, in parallel with the cones $\mathbf{C}_{n,k,\theta}$ of type $(i)$, constructed in~\cite{FTW-1}*{Theorem 1.3}.
By construction, $\tilde{\mathbf{C}}_{n,k,\theta}$ is a regular capillary cone with an isolated singularity at the origin.

When $\theta = \frac{\pi}{2}$, the cone $\mathbf{C}$ is called a \textit{free-boundary minimal cone} and may be doubled across $\Pi$ to produce a complete minimal hypercone $\hat{\mathbf{C}}_{n,k}$ in $\bR^{n+1}$.
The link of this cone is obtained by doubling $\Sigma$ as in~\eqref{eqn:link-equation} across the equator of $\bS^n_+$, thereby producing a closed, $O(n-k) \times O(k) \times \bZ_2$-invariant minimal hypersurface $\mathbf{S}_{n,k}$ of the sphere, which is topologically $\bS^{n-k-1}\times \bS^{k-1} \times \bS^1$.
This symmetry reduces $\mathbf{S}_{n,k}$ to a one-dimensional profile curve in the orbit space (a two-dimensional spherical wedge), while the property $[t_1, t_2] \subset (0,1)$ ensures that the profile curve does not intersect the symmetry axis.
Therefore, it is diffeomorphic to $\bS^{n-k-1} \times \bS^{k-1} \times \bS^1$: the first two factors arise from the group orbits, while the $\bS^1$-factor comes from closing the interval $[t_1, t_2]$ into a $\bZ_2$-invariant loop by doubling across the equator.
Notably, the cone $\hat{\mathbf{C}}_{n,k} = C( \mathbf{S}_{n,k})$ has an $O(n-k) \times O(k) \times \bZ_2$ symmetry that does not integrate to the full $O(n-k) \times O(k+1)$-invariance of the Lawson cone, and carries different topological structure in its link.
In particular, $\hat{\mathbf{C}}_{n,k}$ is not isoparametric.

\subsection{Analysis of the capillary equation}
The existence of the cones $\tilde{\mathbf{C}}_{n,k,\theta}$ will be obtained by analyzing the capillary equation~\eqref{eqn:ode-star} and its rescaled version~\eqref{eqn:capillary-ODE}, for general $\lambda \in [0,\infty]$.

\begin{figure}[t]
  \centering
  \includegraphics[scale = 0.8]{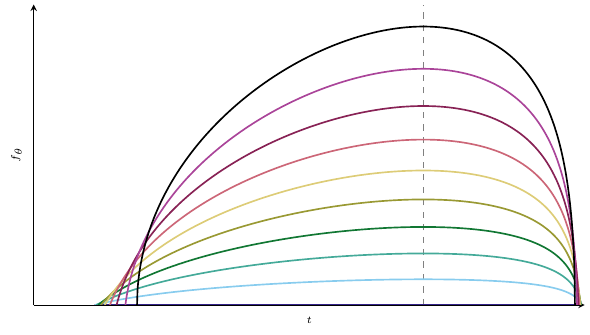}
  \includegraphics[scale = 0.8]{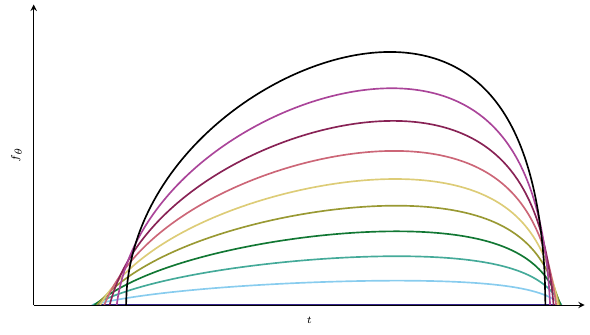}
  \caption{The figures show the graphs of the functions $f_\theta$ for angles $\theta$ ranging from $\ve$ to $\tfrac{\pi}{2}$, with the free-boundary profile $f_{\frac{\pi}{2}}$ in black.
  The left figure illustrates $(n,k) = (4,2)$, when each profile function is symmetric under $t \leftrightsquigarrow \sqrt{1-t^2}$, with a peak at $t = \tfrac{1}{\sqrt{2}}$.
  The right figure illustrates $(n,k) = (5,2)$, which does not have such a symmetry.
  As $\theta \downarrow 0$, the rescaled functions $\frac{1}{\theta} (\rho f_{\theta})$ converge in $C^{\infty}_{\text{loc}}$ to a homogeneous solution of the one-phase problem.}
  \label{fig:interpolations}
\end{figure}

\begin{lemma}\label{lemma:change-of-variables}
For given $(n,k)$, the function $f(t)$ solves equation~\eqref{eqn:capillary-ODE} if and only if the function $f (\sqrt{1-t^2})$ solves this equation for $(n,n-k)$.
Any capillary cones produced by solving this equation with the appropriate free boundary are isometric via the involution $t \leftrightsquigarrow \sqrt{1-t^2}$; in particular, they form the same contact angle.
\end{lemma}
\begin{proof}
The equation~\eqref{eqn:capillary-ODE} is invariant under rewriting $\bR^n = \bR^{n-k}_x \times \bR^k_y$ with coordinate $\sqrt{1-t^2} = \frac{|x|}{\sqrt{|x|^2 + |y|^2}}$, instead of $t = \frac{|y|}{\sqrt{|x|^2 + |y|^2}}$.
Indeed, setting $s := \sqrt{1-t^2}$ and $\hat{f}(s) := f(\sqrt{1-s^2}) = f(t)$, we find that $\frac{dt}{ds} = - \frac{s}{t}$ and $\hat{f}'(s) = - \frac{s}{t} f'(t)$.
Moreover, 
\begin{align*}
    A_{n,k}(t) f'(t) &= - \frac{t}{s} \frac{1}{t} \Bigl( t^2 - \frac{k-1}{n-2} \Bigr) \hat{f}'(s) = \frac{1}{s} \left( s^2 - 1 +\frac{k-1}{n-2} \right) \hat{f}'(s) = A_{n,n-k}(s) \hat{f}'(s), \\
    \hat{f}''(s) &= - \frac{d}{ds} \Bigl( \frac{s}{t} \Bigr) f'(t) - \frac{s}{t} \frac{d}{ds} (f'(t)) = - \frac{f'}{t^3} + \frac{s^2}{t^2} f'',
\end{align*}
using the equalities $(\frac{s}{t})' = \frac{t + t^{-1} s^2 }{t^2} = t^{-3}$ and $\frac{d}{ds} (f') = f'' \frac{dt}{ds} = - \frac{s}{t} f''$.
We may therefore express
\begin{align*}
& (1-t^2) f'' = s^2 \bigl( \tfrac{t^2}{s^2} \hat{f}'' - \tfrac{1}{s^3} \hat{f}' \bigr) = t^2 \hat{f}'' - s^{-1} \hat{f}', \qquad f - tf' = \hat{f} + s^{-1} t^2 \hat{f}', \\
& (1-t^2) f'' + (f - tf') = t^2 \hat{f}'' + \hat{f} - s \hat{f}' = (1-s^2) \hat{f}'' + ( \hat{f} - s \hat{f}'), \\
& (1-t^2) \frac{\lambda (f')^2}{1+ \lambda f^2} = s^2 \frac{ \lambda( \frac{t^2}{s^2} (\hat{f}')^2)}{1+ \lambda \hat{f}^2} = t^2 \frac{\lambda (\hat{f}')^2}{1 + \lambda \hat{f}^2} = (1-s^2) \frac{\lambda (\hat{f}')^2}{1 + \lambda \hat{f}^2}. 
\end{align*}
Combining these relations shows that $\hat{f}(s)$ solves the equation~\eqref{eqn:capillary-ODE} in the variable $s$, for the pair $(n,n-k)$, if and only if $f(t)$ does for $(n,k)$.
Under this transformation, a zero $t_i$ of $f$ becomes a zero $s_i$ of $\hat{f}$, with $s_i = \sqrt{1 - t_i^2}$.
Then, the symmetry $\hat{f}'(s) = - \frac{s}{t} f'(t)$ implies that 
\[
(1-s_i^2) \hat{f}'(s_i)^2 = t_i^2 \hat{f}'(s_i)^2 = s_i^2 f'(t_i)^2 = (1-t_i^2) f(t_i)^2
\]
hence the resulting capillary angles $\theta_i = \on{arctan} \sqrt{ (1 - t_i^2) f'(t_i)^2 }$ are the same for $f(t)$ and $\hat{f}(s)$.
The resulting cones are therefore isometric, as claimed.
\end{proof}

We now establish some fundamental properties of solutions of equation~\eqref{eqn:capillary-ODE}, cf.~\cite{FTW-1}*{\S 3}.
We introduce the quantity
    \begin{equation}\label{eqn:hf-for-lemma}
    h_f (t):= f - A(t) f'.
    \end{equation}
Applying the equation to $f''$ from~\eqref{eqn:capillary-ODE}, we see that $h_f$ satisfies the ODE
\begin{equation}\label{eqn:h'-abbrevaited}
    h'_f = A(t) S(t) h_f - \tfrac{\alpha(1-\alpha)}{t^2 (1-t^2)} f' \, .
\end{equation}
where we denoted $S(t) := \frac{n-1}{1-t^2} + (n-2) \frac{\lambda (f')^2}{1+\lambda f^2} > 0$.
Define $\alpha := \frac{k-1}{n-2}$ (meaning $A(t) = t - \alpha t^{-1}$) and $\psi(t) := \sqrt{|t^2 - \alpha|}$, so $\frac{\psi'}{\psi} = \frac{1}{A}$ for $t \in [0,1] \setminus \{ \sqrt{\alpha} \}$.
Then, the function $h_f = f - \frac{\psi}{\psi'} f'$ satisfies
\begin{equation}\label{eqn:integrating-factor-f-psi}
\left( \frac{f}{\psi} \right)' = \frac{f' - \frac{\psi'}{\psi} f}{\psi} = - \frac{h_f}{A \psi}.
\end{equation}
The relation~\eqref{eqn:h'-abbrevaited} also has the following consequence: the function $\frac{f'}{h_f}$ satisfies the Riccati equation
\begin{equation}\label{eqn:riccati-for-ratio}
\begin{split}
    \Bigl( \frac{f'}{h_f} \Bigr)' = \frac{\alpha(1-\alpha)}{t^2(1-t^2)} \Bigl( \frac{f'}{h_f} \Bigr)^2 + \left( \frac{\alpha}{t(1-t^2)} - A(t) S(t) \right) \frac{f'}{h_f} - S(t).
\end{split}
\end{equation}

\begin{proposition}\label{prop:unifying-proposition}
    Let $f$ be a positive solution of equation~\eqref{eqn:capillary-ODE} on an interval $(t_1,b) \subset (0,1)$ with $\{ f(t_1) \geq 0, f'(t_1) > 0\}$.
    Then the following properties hold:
    \begin{enumerate}[$(i)$]
        \item If $f$ is not strictly increasing for all time, then it has a critical point $t_c$. 
        Additionally, $h_f>0$ on $[t_1, b]$ and $h'_f < 0$ for $t < \min \{ \sqrt{\frac{k-1}{n-2}}, t_c \}$ and $h'_f > 0$ for $t > \max \{ \sqrt{\frac{k-1}{n-2}}, t_c \}$.
        
        \item If $f(t_2) = 0$, then $t_2 > \sqrt{\frac{k-1}{n-2}}$.
        If $f(t_1) = f(t_2) = 0$, then $t_1 < \sqrt{\frac{k-1}{n-2}} < t_2$.
        
        \item The function $f$ is strictly increasing if and only if $h_f$ has a zero $t_h$, and $h_f, h'_f < 0$ for $t \geq t_h$.

    \end{enumerate}
\end{proposition}
\begin{proof}
Working with the rescaled function $f = \sqrt{\lambda} f_{\lambda,a}$, we may set $\lambda=1$ in~\eqref{eqn:capillary-ODE}.
At a zero $t_*$ of $f'$, the relation~\eqref{eqn:capillary-ODE} forces $(1-t_*^2) f''(t_*) = - (n-1) f(t_*) < 0$.
If $f'$ had at least two successive zeros $\tau_1 < \tau_2$, then $f''(\tau_1) < 0$ would imply that $f' < 0$ for small $t - \tau_1 > 0$, whereby $f''(\tau_2) \geq 0$.
This would contradict $(1-\tau_2^2) f''(\tau_2) = - (n-1) f(\tau_2) < 0$, so $f'$ has at most one zero, at $t = t_c$.

Any non-trivial solution $f$ must have $f'(t_*) \neq 0$ whenever $f(t_*) = 0$ by ODE uniqueness, since $f \equiv 0$ solves~\eqref{eqn:capillary-ODE}.
    Using the definition of $h_f$ and equation~\eqref{eqn:hf-for-lemma}, the same property holds for the function $\frac{f}{\psi}$ from~\eqref{eqn:integrating-factor-f-psi}, hence $h(t_*) = h'(t_*) = 0$ cannot occur at some point $t_*$, unless $f \equiv 0$.

    We now examine forward intervals $[t_c,b]$ with $f'(t_c) = 0$, and will use their properties to study the situation when $f$ attains a critical point.
    Suppose $t_2$ is the first zero of $f$ to the right of $t_c$, if it exists, so $f(t_2) = 0$ and $f'(t_2) < 0$.
    From the previous step, we know $f'<0$ on $[t_c, t_2]$.
    Since $f'(t_c) = 0$, we have $h(t_c) = f(t_c) > 0$ and $h > 0$ for small $t - t_c > 0$.
    If $h > 0$ fails in $[t_c,b)$, there is a first zero $\tau_1$ with $h(\tau_1) = 0$ and $h'(\tau_1) < 0$.
    Using the definition of $h$ together with~\eqref{eqn:hf-for-lemma}, we obtain
    \begin{equation}\label{eqn:f-h'-properties}
        f(\tau_1) = A(\tau_1) f'(\tau_1), \qquad h'(\tau_1) = - \tfrac{\alpha(1-\alpha)}{\tau_1^2 (1 - \tau_1^2)} f'(\tau_1).
    \end{equation}
    Since $\alpha = \frac{k-1}{n-2} \in (0,1)$, this implies $\on{sgn} f'(\tau_1) = - \on{sgn} h'(\tau_1) > 0$.
    Since $f'<0$ on $[t_c, t_2]$, we must have $\tau_1 > t_2$, so $h>0$ on $[t_c, t_2]$.
    Therefore, 
    \[
    0 < h(t_2) = f(t_2) - A(t_2) f'(t_2) = |f'(t_2)| \, A(t_2) \implies A(t_2) > 0 \implies t_2 > \sqrt{\tfrac{k-1}{n-2}} ,
    \]
    which means that $f$ cannot become zero before $\sqrt{\frac{k-1}{n-2}}$.
    A direct adaptation of~\cite{FTW-1}*{Lemma 3.4} now proves that the properties $h>0$ and $f'<0$ are preserved on $[t_2, b)$; in particular, $f'<0$ on $(t_2, b)$ and $f'$ has at most one zero in $(t_c, b)$.
    Moreover, $f<0$ and $\psi$ is strictly increasing on $[t_2,b)$, while $\frac{f}{\psi}$ is strictly decreasing, so $f$ is a strictly decreasing negative function.
    We conclude that $f'<0$ and $h > 0$ on $[t_c,b]$.
    Therefore, $f$ has at most one critical point, which is a maximum.

    We now prove that $h>0$.
    By definition, $h(t_1) = - A(t_1) f'(t_1) > 0$ and $h(t_c) = f(t_c) > 0$, so if $h > 0$ failed, then $h$ would have at least two sign changes $\tau_1 < \tau_2 \in (t_1, t_c)$.
    Observe that~\eqref{eqn:h'-abbrevaited} forces $\on{sgn} h'(t_*) = - \on{sgn} f'(t_*)$ at a zero of $h$.
    Let $\tau_1$ be the first zero of $h$ and $\tau_2$ be the last zero of $h$ in $(t_1,t_c)$, so that $h'(\tau_1) < 0 $ and $h'(\tau_2) > 0$ with strict inequality, by ODE uniqueness.
    Then $f'(\tau_2) < 0 < f'(\tau_1)$, contradicts the fact that $f'>0$ on $[t_1, t_c]$; we conclude that $h>0$ on $[t_1,t_c]$.

   In the special case where we provide initial data $\{ f(\sqrt{\frac{k-1}{n-2}}) = 0, f'(\sqrt{\frac{k-1}{n-2}}) = a > 0 \}$, $f$ will not have a critical point.
    Suppose otherwise, and $f$ had a critical point at some point $t_c > \sqrt{\frac{k-1}{n-2}}$.
    Then, the symmetry of Lemma~\ref{lemma:change-of-variables} shows that $\hat{f}(t) := f(\sqrt{1-t^2})$ is a forward solution of the corresponding equation~\eqref{eqn:capillary-ODE} for $k \leftrightsquigarrow n-k$ with $\hat{f}'(\hat{t}_c) = 0$ and $\hat{f}( \sqrt{\frac{n-k-1}{n-2}}) = 0$.
    However, applying the previous step to $\hat{f}(t)$, for $k \leftrightsquigarrow n-k$, would imply that $\hat{f}(\sqrt{\frac{n-k-1}{n-2}}) > 0$; this produces a contradiction.
    Therefore, $f'>0$ on the domain of $f$ in this case.
    
    At a zero of $f$, $h_f'$ and $f'$ must have opposite signs.
    To see this, we consider the relation~\eqref{eqn:h'-abbrevaited} at $t_*$, where $f(t_*) = 0$ and $h(t_*) = - A(t_*) f'(t_*) > 0$ (due to $t_* < \sqrt{\frac{k-1}{n-2}}$ and $f'(t_*) > 0$) to obtain
    \begin{align*}
    h'(t_*) &= A(t_*) \Bigl( \tfrac{n-1}{1-t_*^2} + (n-2) f'(t_*)^2 \bigr) (- A(t_*)) f'(t_*) - \tfrac{\alpha(1-\alpha)}{t_*^2 (1-t_*^2)} f'(t_*) \\
    &= - f'(t_*) \left[ A(t_*)^2 \bigl( \tfrac{n-1}{1-t_*^2} + (n-2) f'(t_*)^2 \bigr) + \tfrac{\alpha(1-\alpha)}{t_*^2 (1 - t_*^2)} \right].
    \end{align*}
    The bracketed quantity is strictly positive due to $\alpha = \frac{k-1}{n-2} \in (0,1)$, hence $\on{sgn} h'(t_*) = - \on{sgn} f'(t_*)$.
    To obtain the monotonicity of $h$, we use $h>0$ on $[t_1,b]$ to see that equation~\eqref{eqn:h'-abbrevaited} has $h'$ both strictly negative for $t < \min \{ \sqrt{\frac{k-1}{n-2}}, t_c \}$, and strictly positive for $t > \max \{ \sqrt{\frac{k-1}{n-2}}, t_c \}$.
    This proves $(i)$.

    To see $(ii)$, suppose that $f$ vanishes on $\{t_1,t_2\}$, so there is a unique critical point $t_c\in (t_1,t_2)$.
    We showed $t_2 > \sqrt{\frac{k-1}{n-2}}$ and applying the symmetry of Lemma~\ref{lemma:change-of-variables} shows that the function $\hat{f}(t) := f(\sqrt{1-t^2})$ solves the corresponding equation~\eqref{eqn:capillary-ODE} for $k \leftrightsquigarrow n-k$, with zeros at $\hat{t}_1 := \sqrt{1-t_2^2}$ and $\hat{t}_2 := \sqrt{1 - t_1^2}$.
    Therefore, $\hat{t}_2 > \sqrt{\frac{n-k-1}{n-2}}$, which implies that $t_1 = \sqrt{1 - \hat{t}_2} < \sqrt{\frac{k-1}{n-2}}$, proving $(ii)$.
    
    For $(iii)$, the first direction of this result follows from the previous steps.
    Suppose that $f$ is strictly increasing.
    If $h$ has a sign change at $t_h > t_1$, then the relation~\eqref{eqn:h'-abbrevaited} implies that $\on{sgn} h'(t_h) = - \on{sgn} f'(t_h) < 0$, meaning that $h$ can only cross downward.
    Thus, either $h<0$ for $t>t_1$, or $h$ has a unique sign change $t_h$, and $h<0$ for $t > t_h$.
    In the latter situation, 
    \[
    0 = h(t_h) = f(t_h) - A(t_h) f'(t_h) > - A(t_h) f'(t_h) \implies t_h > \sqrt{\tfrac{k-1}{n-2}}.
    \]
    Finally, for $t \geq t_h > \sqrt{\frac{k-1}{n-2}}$, we have $h <0 $ and $(-f')<0$, hence using the relation~\eqref{eqn:h'-abbrevaited}, as in part $(i)$, shows that $h'<0$ for $t \geq t_h$.
\end{proof}
Taking $\lambda=0$ in equation~\eqref{eqn:capillary-ODE} recovers the equation $\cL_{n,k}f = 0$ from~\eqref{eqn:legendre-form}, for which the function $U = \rho f(t)$ is harmonic.
Moreover, $|\nabla U|^2 = f^2 + (1-t^2) (f')^2$, hence $|\nabla U|$ is constant on $\{ U > 0 \}$ if and only if the expression $c = (1-t_i^2) f'(t_i)^2$ is constant at the zeros of $f$, in which case $\frac{1}{\sqrt{c}} U$ is a solution of the one-phase problem.
In a related paper~\cite{six-way}, we examine the corresponding one-phase solutions of this type and prove that for any $2 \leq k \leq n-2$, there exist exactly two $O(n-k) \times O(k)$-invariant solutions of the one-phase problem: one for which $f$ has positive phase the interval $[-t_{n,k}, t_{n,k}]$, and one for which $f$ has positive phase an interval $[ \tilde{t}_1, \tilde{t}_2] \subset (0,1)$ (with the dependence of the $\tilde{t}_i$ on $n,k$ suppressed).
The first family is considered in~\cites{hong-singular , FTW-stability-one-phase}, whereas the second family of solutions obtained in~\cite{six-way} is novel.

In the first case, the solution is a rescaling of the hypergeometric function 
\[
f_{n,k}(t) := {}_2 F_1\left(\frac{n-1}{2} ,- \frac{1}{2} ; \frac{k}{2} ;t^2\right), \qquad \cL_{n,k} f_{n,k} = 0, \qquad f_{n,k}(0) = 1, \qquad f'_{n,k}(0) = 0,
\]
with a zero at $t_{n,k} > 0$.
In the second case, we denote the corresponding solution by $\tilde{f}_{n,k}(t)$ vanishing at $\tilde{t}_1,\tilde{t}_2 > 0$ with $\tilde{U}_{n,k} := \rho \tilde{f}_{n,k}(t)$, so that $(1-\tilde{t}_1^2) \tilde{f}'(\tilde{t}_1)^2 = (1-\tilde{t}_2^2) \tilde{f}'(\tilde{t}_2)^2 = 1$.
Since $\cL f_{n,k} = \cL\tilde{f}_{n,k} = 0$, the Sturmian theory implies that $t_{n,k} \in (\tilde{t}_1, \tilde{t}_2)$.

We will again use a linear barrier for the solutions of equation~\eqref{eqn:capillary-ODE}, as done in~\cite{FTW-1}*{\S~3}.

\begin{lemma}\label{lemma:harmonic-barrier}
    There exists a $\hat{t}_{n,k} \in (\tilde{t}_1, \sqrt{\frac{k-1}{n-2}})$ such that the unique solution of equation $\cL_{n,k} f = 0$ with initial data $\{ f(t_1) = 0, f'(t_1) = 1\}$ has another zero if $t_1 \leq \hat{t}_{n,k}$, and stays positive if $t_1 > \hat{t}_{n,k}$. 
\end{lemma}
\begin{proof} For $t_1$ such that a solution $f$ with a second zero exists, let $t_2 := \tau(t_1)$ denote the second zero, so $\sqrt{\frac{k-1}{n-2}} \in (t_1, t_2)$ by Proposition~\hyperref[prop:unifying-proposition]{\ref{prop:unifying-proposition}$(ii)$}.
Using the form~\eqref{eqn:legendre-form}, we may write
\[
f(t_1) = f(t_2) = 0, \qquad - \tfrac{1-t^2}{p(t)} (pf')' = (n-1) f,
\]
so $(n-1)$ is a Dirichlet eigenvalue of the Sturm-Liouville operator $\cA u := - \frac{1-t^2}{p(t)} (p u')'$ on $(t_1, t_2)$.
In fact, this is the first Dirichlet eigenvalue, meaning that $\lambda_1(t_1, t_2) = n-1$.
The monotonicity of the Dirichlet eigenvalue (together with Sturmian theory) implies that $\lambda_1(t_1,t_2)$ is increasing in $t_1$ and decreasing in $t_2$.
In particular, for any two admissible $t_1 < \bar{t}_1$, this result also shows that $\tau(t_1) < \tau(\bar{t}_1)$.
The result of Proposition~\hyperref[prop:unifying-proposition]{\ref{prop:unifying-proposition}$(ii)$} shows that the solution starting at $\sqrt{\frac{k-1}{n-2}}$ remains positive, so this is not an admissible initial point for a solution that returns to zero; consequently, $\lambda_1 ( \sqrt{\frac{k-1}{n-2}},1) > n-1$, and we can find some $\hat{t}_{n,k} < \sqrt{\frac{k-1}{n-2}}$ such that $\lambda_1 ( \hat{t}_{n,k},1) = n-1$.
Therefore,
\[
t_1 \leq \hat{t}_{n,k} \implies \lambda_1(t_1,1) \leq n-1, \qquad t_1 > \hat{t}_{n,k} \implies \lambda_1(t_1,1) > n-1,
\]
which makes $\lambda_1(t_1,\tau(t_1)) = n-1$ possible if and only if $t_1 \leq \hat{t}_{n,k}$.
\end{proof}

We now obtain some of the key properties behind the proof of Theorem~\ref{thm:MinimalSurfaceExistence}.
The crucial observation is that, somewhat surprisingly, any candidate capillary profiles must satisfy strong uniform bounds.
In particular, solutions of the form $\{ f(t_1) = 0, f'(t_1) = M \}$ with arbitrarily large slope $M$ converge as $M \to \infty$, allowing us to define and study \textit{vertical} shots $\{ f(t_1) = 0, f'(t_1) = \infty\}$.

\begin{lemma}\label{lemma:bound-at-sqrt}
Any solution $f$ of equation~\eqref{eqn:capillary-ODE} with two zeros in $(0,1)$ satisfies $f(\sqrt{\frac{k-1}{n-2}}) \leq \frac{2}{\sqrt{n}} \lambda^{- \frac{1}{2}}$.
\end{lemma}
\begin{proof}
    Working with the rescaled function $\sqrt{\lambda} f$, which solves~\eqref{eqn:ode-star}, we need only prove the result when $\lambda=1$.
    As in~\cite{FTW-1}*{\S 3}, we consider the quantity $\Psi(t) := f (f - Af') - \frac{1}{n-2}$ for a solution of equation~\eqref{eqn:ode-star}.
    The results of~\cite{FTW-1}*{\S 3.3} show that $\Psi'(t_{\alpha}) = 0$, due to
    \begin{align*}
    \Psi'(t) &= A(t) \bigl( t^{-1} ff' - (f')^2 -ff''), \qquad \Psi''( \sqrt{\tfrac{k-1}{n-2}}) = 2 \bigl( t^{-1} f f' - (f')^2 - ff'' \bigr)|_{t = t_{\alpha}}.
    \end{align*}
    For brevity, we set $F:= f( t_{\alpha})$ and $M := f'(t_{\alpha})$.
    Using $A( t_{\alpha}) = 0$, we solve for $f''(t_{\alpha})$ in equation~\eqref{eqn:ode-star}, then apply it to the computation of $\Psi''( t_{\alpha})$ to obtain
    \begin{equation}\label{eqn:Psi-''-expression}
    \tfrac{1}{2} F^{-2}  \Psi''(t_{\alpha}) = \frac{(n-3) F^2 - 1}{1+F^2} \Bigl( \frac{M}{F} \Bigr)^2 - \frac{(2k-n)}{(n-k-1)\sqrt{\frac{k-1}{n-2}}} \frac{M}{F} + \frac{(n-1)(n-2)}{n-k-1} .
    \end{equation}
    For $F > \frac{1}{\sqrt{n-3}}$, the expression of $\Psi''( t_{\alpha}) $ is a convex quadratic $a(F) (\frac{M}{F})^2 + b \frac{M}{F} + c$ in $\frac{M}{F}$, so it attains the minimum value $c - \frac{b^2}{4 a(F)}$.
    Let us write
    \begin{align*}
        4 ac - b^2 &= 4 a \tfrac{(n-1)(n-2)}{n-k-1} - \tfrac{(2k-n)^2}{(n-k-1)^2 \frac{k-1}{n-2}} = \tfrac{n-2}{(n-k-1)^2 (k-1)} \bigl[ 4 a (n-1)(n-k-1)(k-1) - (2k-n)^2 \bigr]
    \end{align*}
    which is non-negative if and only if $a \geq \max_{2 \leq k \leq n-2} \frac{(2k-n)^2}{4(n-1)(n-k-1)(k-1)}$.
    The function $x \mapsto \frac{(2x-m)^2}{x(m-x)}$ is a convex parabola on $[0,m]$, symmetric over the midpoint $\frac{m}{2}$, hence it attains its maximum over $[1,m-1]$ at $x=1$.
Applying this observation for $x=k-1 \in [1,n-3]$ and $m = n-2$, we find that $4 ac - b^2 \geq 0$ provided that $a \geq \frac{(n-4)^2}{4(n-1)(n-3)}$.
Finally, taking $F \geq \frac{2}{\sqrt{n}}$, we compute
\[
a(F) - \frac{(n-4)^2}{4(n-1)(n-3)} =  \frac{(n-2)(n-4)(11n - 26)}{4(n-3)(n-1)(n+4)} \geq 0, \qquad \text{for } \; n \geq 4.
\]
Thus, $\Psi''( t_{\alpha}) \geq 2 F^2 \bigl[ \min_M a(F) (\frac{M}{F})^2 + b \frac{M}{F} + c \bigr] \geq 0$ holds for $F = f( t_{\alpha}) \geq \frac{2}{\sqrt{n}}$ and any $M \in \bR$.

Since the properties of $f$ are invariant under the involution $( k ,f(t) )\leftrightsquigarrow (n-k, f(\sqrt{1-t^2}) )$, which preserves solutions by Lemma~\ref{lemma:change-of-variables}, we may assume without loss of generality that $f'(t_{\alpha}) \leq 0$, since $f'(t_{\alpha}) > 0$ would imply that $\hat{f}'(t_{\hat{\alpha}}) < 0$ for $\hat{f}(t) := f(\sqrt{1-t^2})$ and $t_{\hat{\alpha}}:= \sqrt{\frac{n-k-1}{n-2}}$.
Having assumed that the function $f$ attains two zeros, we know from Proposition~\ref{prop:unifying-proposition} that these occur with $t_1 < t_{\alpha} < t_2$, while $h = f - Af'$ is strictly positive on $[t_1,t_2]$ and $f'<0$ on $[t_{\alpha},t_2]$.
We now recall the arguments of~\cite{FTW-1}*{Proposition 3.16}: at a critical point $t_*$ of $\Psi$, we have 
\begin{equation}\label{eqn:Psi-at-critical-point}
\Psi'(t_*)= 0 \implies \Psi''(t_*) = \frac{2 (n-2)^2 (-f') (f - A f')}{(1-t_*^2) (1+f^2)} \left( 1 + \frac{(1-t_*^2) (f')^2}{1 + f^2} \right) A (t_*)  \Psi(t_*).
\end{equation}
Since $(-f')(f - Af') > 0$ on $[t_{\alpha}, t_2]$, equation~\eqref{eqn:Psi-at-critical-point} implies that $\on{sgn} \Psi''(t_*) = \on{sgn} \Psi(t_*)$ at a critical point $t_* > t_{\alpha}$, due to $A(t_*) > 0$.
Therefore, any solution of equation~\eqref{eqn:ode-star} with $\{ \Psi(t_{\alpha}) > 0, \Psi''( t_{\alpha}) \geq 0 \}$ must have $\Psi$ strictly positive and strictly increasing for $t > t_{\alpha}$, so $f$ cannot become zero after $t_{\alpha}$.
For $F \geq \frac{2}{\sqrt{n}}$, we find $\Psi( t_{\alpha}) = F^2 - \frac{1}{n-2} \geq \frac{3n-8}{n(n-2)} > 0$ for $n \geq 4$, so the above arguments imply that $f$ cannot reach a zero $t_2 > t_{\alpha}$; this produces a contradiction, proving the claimed bound.
\end{proof}

\begin{lemma}\label{lemma:slope-infinity-solution}
    Fix some $t_1 < \sqrt{\frac{k-1}{n-2}}$ and let $f^{(M)}$ denote the solution of equation~\eqref{eqn:capillary-ODE}, for $\lambda>0$, with initial data $\{ f^{(M)}(t_1) =0, f^{(M)}{}'(t_1) = M \}$.
    Then, there exists a solution $f^{(\infty)}$ of equation~\eqref{eqn:capillary-ODE} defined on an interval $[t_1,b]$ such that $f^{(\infty)}(t_1) = 0$ and $f^{(M)} \to f^{(\infty)}$ in $C^{\infty}_{\textup{loc}}(t_1,b)$ as $M \to \infty$.
    We call $f^{(\infty)}$ the solution of equation~\eqref{eqn:capillary-ODE} with initial data $\{ f^{(\infty)}(t_1) =0 , f^{(\infty)}{}'(t_1) = \infty\}$.
\end{lemma}
\begin{proof}
If $\lambda=0$, the result is clear from Lemma~\ref{lemma:harmonic-barrier}, so we consider $\lambda>0$ and rescale each $f$ by $\sqrt{\lambda}$, if needed, to assume that $\lambda=1$ and $f$ solves equation~\eqref{eqn:ode-star}.
The function $A(t)$ is strictly increasing, so $A(t_1) < A(t) < A(T_1) < 0$ for any $t \in (t_1,T_1)$ with $T_1 < \sqrt{\frac{k-1}{n-2}}$.
Therefore,
\[
f - Af' \geq f - A(T_1) f', \qquad 1 + (1-t^2) \tfrac{(f')^2}{1+f^2} > 1 + (1-t^2) (f')^2 (1-f^2)
\]
while $f' \geq 0$; the latter inequality is due to $\frac{1}{1+x}>1-x$.
Combining these bounds in equation~\eqref{eqn:ode-star} produces the inequality
\begin{equation}\label{eqn:f''-inequality}
(1-t^2) f'' + (f - tf')+ (n-2) \bigl[ 1 + (1-t^2) (f')^2 (1-f^2) \bigr](f - A(T_1) f') < 0.
\end{equation}
Given a point $b > t_1$, we consider the following property:
\begin{equation}\label{eqn:uniform-boundedness-property}
    \text{for any small} \; 0 < \ve < \tfrac{b - t_1}{4}, \qquad \limsup_{M \to \infty} \sup_{ [t_1 + \ve, b-\ve] }  \| f^{(M)} \|_{C^2} \leq C(\ve)
\end{equation}
for a constant $C(\ve)$ that is independent of $M$.
On any interval where the property~\eqref{eqn:uniform-boundedness-property} holds, we can now apply the Arzel\`a-Ascoli theorem, together with a standard diagonal argument along the compact exhaustion $[t_1 + \frac{1}{j}, b - \frac{1}{j}]$ (for $j \gg 1$), to deduce that $f^{(M)} \to f^{(\infty)}$ in $C^2_{\text{loc}}(t_1,b)$, and therefore in $C^{\infty}_{\text{loc}}$ because the $f^{(M)}$ are solutions of the elliptic equation~\eqref{eqn:ode-star}.
Moreover, $f^{(\infty)} \in C^{\infty}(t_1,b) \cap C^{\frac{1}{2}}([t_1,b))$, and taking limits makes $f^{(\infty)}{}'(t_1) = \infty$ by construction.

We now prove that there exists some $L_1 \in (t_1, \sqrt{\frac{k-1}{n-2}})$, depending only on $(n,k,t_1)$, such that the property~\eqref{eqn:uniform-boundedness-property} holds on $[t_1,L_1]$.
Dropping the positive term
\[
f + (n-2) (f - A(T_1) f') + (n-2) (1-t^2) \tfrac{(f')^2}{1+f^2} f
\]
from equation~\eqref{eqn:ode-star}, we may further simplify the inequality~\eqref{eqn:f''-inequality} into
\begin{equation}\label{eqn:further-rearrangement}
    f'' < c f' - (n-2) \, |A(T_1)| \, (1-f^2) (f')^3
\end{equation}
as long as $f' > 0$.
Here, we divided by $1-t^2$ and used $\frac{t}{1-t^2} \leq c := \frac{\sqrt{(k-1)(n-2)}}{n-k-1}$ for $t \leq \sqrt{\frac{k-1}{n-2}}$.
On the interval $[t_1,L]$ where $f \leq \sqrt{1-\delta}$, for some $\delta \in (0,1)$, we use $1-f^2 \geq \delta$ to see that the function $\tilde{y}(t) := \sqrt{(n-2) \, |A(T_1)| \, \delta} f'(t)$ satisfies $\tilde{y}'(t) < c \cdot \tilde{y}(t) - \tilde{y}(t)^3$ due to~\eqref{eqn:further-rearrangement}.
The associated Bernoulli equation $y'(t) = c\cdot y(t) - y(t)^3$ has the explicit solution $y(t) = \frac{\sqrt{a} e^{c(D_1 + t)}}{\sqrt{e^{2c(D_1+t)}}-1}$.
Imposing the initial condition $\lim_{t \downarrow 0} y(t) = \infty$ resolves $D_1 = 0$ and $y(t) = \frac{\sqrt{c} e^{ct}}{\sqrt{e^{2ct}-1}}$.
Applying the ODE comparison principle, we conclude that if $f' \geq 0$ on $[t_1,L]$, then
\begin{equation}\label{eqn:ODE-comparison}
    f' \Bigl(t_1 + \frac{s}{c(T_1)} \Bigr) \leq \sqrt{\frac{c(T_1)}{(n-2) \, |A(T_1)| \, \delta} } \, \cdot \frac{e^{s}}{\sqrt{e^{2s}-1}} \qquad \text{for } \; s \in \Bigl[0,\frac{L - t_1}{c(T_1)} \Bigr].
\end{equation}
To bound the above terms further, we use $e^{-2s} \leq \frac{1}{1+2s}$ and $1 - e^{-2s} \geq \frac{2s}{1+s}$ to estimate 
\[
\int_0^s \frac{dx}{\sqrt{1 - e^{-2x}}} \leq \int_0^s \sqrt{\frac{1 + 2x}{2x}} \leq \int_0^s \Bigl( \frac{1}{\sqrt{2x}} + \sqrt{\frac{x}{2}} \Bigr) \, dx = \sqrt{2s} \Bigl( 1 + \frac{s}{3} \Bigr).
\]
Let $L_*$ be the first instance such that $f(L_*) = \sqrt{1-\delta^2}$, if this occurs, or $L_* = \sqrt{\frac{k-1}{n-2}}$ otherwise.
Then, we can integrate~\eqref{eqn:ODE-comparison} on $[0, \frac{L_* - t_1}{c}]$ and bound $L_* - t_1 \leq \sqrt{\frac{k-1}{n-2}}$ to obtain
\begin{equation}\label{eqn:key-integral-bound}
\sqrt{1-\delta^2} \leq \sqrt{\frac{c}{(n-2) \, |A(T_1)| \, \delta}} \int_0^{\frac{L_* - t_1}{c}} \frac{dx}{\sqrt{1 - e^{-2x}}} \leq C(n,k) \sqrt{\frac{2(L_* - t_1)}{(n-2) \, |A(T_1)| \,  \delta}} \; ,
\end{equation}
which implies that $L_* - t_1 \geq C(n,k) |A(T_1)| \delta(1-\delta^2)$.
Finally, we may choose $T_1 = \sqrt{\frac{k-1}{n-2}} - \ve$ (for small $\ve>0$, depending on $t_1$) and $\delta = \frac{1}{2}$ to guarantee the existence of an $L_1(n,k,t_1)>t_1$ such that $f(t)<\frac{1}{2}$ on $[ t_1, L_1]$, independently of $f'(t_1)$.
Applying the bound~\eqref{eqn:ODE-comparison}, we conclude that
\[
\sup_M \sup_{[t_1,L_1]} f^{(M)} \leq \tfrac{1}{2}, \qquad \sup_M \sup_{[t_1+\ve,L_1]} f^{(M)}{}' \leq C(\ve) \quad \text{for any } \; \ve>0,
\]
whereby the inequality~\eqref{eqn:further-rearrangement} ensures that $\sup_M \sup_{[t_1,+\ve,L_1]} f^{(M)}{}'' \leq C(\ve)$ as well, uniformly in $M$.
This shows that~\eqref{eqn:uniform-boundedness-property} is satisfied on $[t_1,L_1]$ as desired.

Using the above construction, we now take $\ve>0$ and extend $\tilde{f}^{(\infty)}$ to be the maximal solution of equation~\eqref{eqn:ode-star} with initial condition 
\[
\{ f^{(\infty)}(t_1+\ve) = \lim_{M \to \infty} f^{(M)}(t_1+\ve), \quad f^{(\infty)}{}'(t_1+\ve) = \lim_{M \to \infty} f^{(M)}{}'(t_1+\ve) \}.
\]
The previous discussion shows that $\tilde{f}^{(\infty)} = f^{(\infty)}$ on $[t_1+\ve,b_*]$, where $\ve>0$ is arbitrary and $b_* \leq 1$ is the maximal time of definition of $f^{(\infty)}$.
Since the $f^{(M)}$ and $f^{(\infty)}$ are solutions of~\eqref{eqn:ode-star}, we conclude that $f^{(M)} \to f^{(\infty)}$ in $C^{\infty}_{\text{loc}}(t_1,b_*)$.
\end{proof}

\begin{lemma}\label{lemma:convergence-of-solutions}
    Consider an interval $[d_1,d_2] \subset (0, \sqrt{\frac{k-1}{n-2}})$ and fix $a \in (0,\infty]$.
    Suppose that the solution $f_{t_1}$ of equation~\eqref{eqn:capillary-ODE} with initial data $\{ f(t_1) = 0, f'(t_1) = a \}$ reaches a second zero $\tau(t_1) > \sqrt{\frac{k-1}{n-2}}$ with finite derivative when $t_1 = d_1$.
    Then, either this property holds for all $t_1 \in [d_1, d_2]$, or there exists a $t_* \in (d_1,d_2]$ so that the solution $f_{t_*}$ has a zero $\tau(t_*) > \sqrt{\frac{k-1}{n-2}}$ with $f'_{t_*}(t) \to - \infty$ as $t \uparrow \tau(t_*)$.
\end{lemma}
\begin{proof}
    For $a \in (0,\infty)$, the result follows by the smooth dependence of ODE solutions on the initial parameters.
    For $a = \infty$, the mechanism is heuristically the same, and the analogous property is valid upon formalizing the convergence of solutions with smoothly varying data and vertical initial slope (per Lemma~\ref{lemma:slope-infinity-solution}).    
    Let us examine the local behavior of a solution to equation~\eqref{eqn:capillary-ODE} near a point $t_1$ with $\{ f(t_1) = c_0, f'(t_1) = \pm \infty \}$.
As in Proposition~\hyperref[prop:unifying-proposition]{\ref{prop:unifying-proposition}$(iii)$} and Lemma~\ref{lemma:harmonic-barrier}, we will consider $t$ approaching with $- (t - t_1) A(t_1) > 0$, i.e., $t > t_1$ if $t_1 < \sqrt{\frac{k-1}{n-2}}$ (forward solution) and $t < t_1$ if $t_1 > \sqrt{\frac{k-1}{n-2}}$ (solution with endpoint $t_1$).
As $|f'| \to \infty$ with $f \to c_0$, the factor $1 + (1-t^2) \frac{\lambda (f')^2}{1+\lambda f^2} \sim (1-t_1^2) \frac{\lambda (f')^2}{1 + \lambda c^2}$, and $f - Af' \sim - A(t_1) f'$.
Since $t_1 < 1$, the equation has leading-order balance
\begin{equation}\label{eqn:f''-equation-balancing-blowup}
(1-t_1^2) f'' \sim (n-2) (1-t_1^2) \tfrac{\lambda (f')^2}{1 + \lambda c^2} A(t_1) f' \implies f'' \sim \tfrac{(n-2) \lambda A(t_1)}{1 + \lambda c^2} (f')^3,
\end{equation}
with solution $f'(t)^2 \sim - \frac{1 + \lambda c^2}{2 (n-2) A(t_1)} \frac{1}{t - t_1}$ for $- (t - t_1)A(t_1) > 0$.
Integrating once shows that $f$ has a square-root singularity at the point $t_1$, with 
\begin{equation}\label{eqn:f(t)f'(t)-expansion}
\begin{split}
f(t) &= c_0 + \sqrt{\tfrac{2 (1 + \lambda c_0^2)}{\lambda(n-2) |A(t_1)|}} |t - t_1|^{\frac{1}{2}} + O( |t - t_1|), \\
f'(t) &= \on{sgn}(t-t_1) \sqrt{\tfrac{1 + \lambda c_0^2}{2 \lambda (n-2) |A(t_1)|}} |t - t_1|^{- \frac{1}{2}} + O (1).
\end{split}
\end{equation}
Note that $\on{sgn}(t - t_1) = - \on{sgn} A(t_1)$, by assumption, so $f(t)$ has a series expansion in $\sqrt{|t - t_1|}$, and 
\[
f''(t) = \mp \tfrac{1}{4} \sqrt{\tfrac{2 (1 + \lambda c_0^2)}{\lambda (n-2) |A(t_1)| }} \, |t - t_1|^{- \frac{3}{2}} + O ( |t - t_1|^{- \frac{1}{2}}).
\]
Iterating this discussion produces all the terms in the expansion of $f(t) = \sum_{k=0}^{\infty} c_k(t_1,\lambda,c_0) s^k$ in terms of $s := \sqrt{|t - t_1|}$.
Notably, we observe that the coefficients $c_k(t_1,\lambda,c_0)$ are uniformly bounded on any compact domain of the form $(t_1,\lambda) \in [d_1,d_2] \times [\delta,\infty)$, for any $\delta>0$ and $[d_1,d_2] \subset (0,\sqrt{\frac{k-1}{n-2}})$.

In terms of the function $u := f^2$, the equation~\eqref{eqn:capillary-ODE} is expressible as
\begin{equation}\label{eqn:u-equation}
\begin{split}
0 &= (1-t^2) u'' + (2u - tu') \\
& \quad - (1-t^2) \frac{(u')^2}{2u} \left( 1 + (n-2) \frac{\lambda A u'}{2(1+ u)} \right) + (n-2) \left[  2u - Au' + (1-t^2) \frac{\lambda (u')^2}{2(1+ u)} \right] 
\end{split}
\end{equation}
which is regular across $\{ u = 0 \}$ provided that $u'(t_1)=0$ or $u'(t_1) = - \frac{2}{(n-2)\lambda A(t_1)}$ at the zero.
An entirely analogous property holds for the function $u := (f - c_0)^2$ if $f$ in the general case considered above.
Applying the above uniformity discussion, let us define
\[
f_{t_1, \lambda, c_0} (t) := \text{solution of~\eqref{eqn:capillary-ODE} with initial data } \; \{ f_{t_1,\lambda, c_0}(t_1) = c_0 \;, \; \on{sgn}A(t_1) f'_{t_1,\lambda, c_0}(t_1) = -\infty \}.
\]
Then, we conclude that for any sequence of $(t_1^{(i)},\lambda^{(i)},c_0^{(i)}) \in [d_1,d_2] \times [\delta, \infty) \times [-M,M]$ as above, converging to some $(t_1^{(\infty)}, \lambda^{(\infty)}, c_0^{(\infty)}) \in [d_1,d_2] \times [\delta, \infty) \times [-M,M]$, we have uniform bounds on all the coefficients $c_k(t_1^{(i)}, \lambda^{(i)}, c_0^{(i)})$ of the series expansions $f^{(i)} := f_{t_1^{(i)}, \lambda^{(i)}, c_0^{(i)}}$ and may therefore extract subsequential limits $c_k^{(\infty)} := \lim_{j_k \to \infty} c_k^{(j_k)}$.
By a diagonalization argument, this limit is realizable by a single (sub)sequence of functions $f^{(i)}$.
Moreover, the uniform bounds on the coefficients (also obtained as a consequence of the uniform estimates in Lemma~\ref{lemma:slope-infinity-solution}) imply a uniform lower bound on the radius of convergence, so the series $f^{(\infty)} := \sum_{k=0}^{\infty} c_k(t_1^{(\infty)}, \lambda^{(\infty)}, c_0^{(\infty)}) |t - t_1^{(\infty)}|^{\frac{k}{2}}$ is defined and $C^{\infty} ( t_1^{(\infty)},b) \cap C^{\frac{1}{2}}( [t_1^{(\infty)},b))$ on an interval, with $f^{(i)} \to f^{(\infty)}$.
The corresponding functions $u^{(i)} := (f^{(i)})^2$ solve equation~\eqref{eqn:u-equation}, so they satisfy $u^{(i)} \to u^{(\infty)}$ in $C^{\infty}$ on any common interval where $\sup_i \|u^{(i)}\|_{C^2} < \infty$.

Finally, we may suppress the previous notation to $f_{t_1} := f_{t_1,\lambda,0}$ for $c_0 = 0$ and $\lambda \in (0,\infty]$ fixed.
As in the case of $a < \infty$, we now define the instance $t_* \in [d_1,d_2]$ by
\begin{equation}\label{eqn:t-star-definition}
    t_* := \sup \left\{ t_1 \in [d_1,d_2] : f_{t_1} \; \text{reaches another zero} \right\} \, .
    \end{equation}
It follows by the above discussion applied to $f_{t_1}$ directly, or to $u_{t_1} := f_{t_1}^2$, that the smooth dependence of ODE solutions on the initial parameters holds for this family.
Notably, if some $f_{t_1}$ reaches another zero at $\tau(t_1) > \sqrt{\frac{k-1}{n-2}}$ with $| f'_{t_1}(\tau(t_1)) | < \infty$, then so does $f_{t_1+\delta}$, for $\delta \in (-\delta_0, \delta_0)$ sufficiently small.
By continuity of the map $t_1 \mapsto \tau(t_1)$, we also find $f_{t_*}(\tau(t_*)) = 0$.
We conclude that either $t_* = d_2$, or $t_* < d_2$ and at the limit endpoint $t_*$, the solution $f_{t_*}$ reaches zero at $\tau(t_*)$ with $f'_{t_*} (\tau(t_*)) = - \infty$.
\end{proof}

The next key step is to produce a small window near $\sqrt{\frac{k-1}{n-2}}$, from which no solutions, \textit{shots}, reach a second zero; see Proposition~\ref{prop:small-breather}.
The following problem 
 \begin{equation}\label{eqn:phi-limit-equation}
    \phi'' + c \frac{(\phi')^2}{1 + \phi^2} ( \phi - 2 s \phi') = 0, \qquad \{ \phi(-1) = 0, \; \phi'(-1) = \infty \}
\end{equation}
appears as the limit problem in the rescaled variable $s = \epsilon_i^{-1} (t - \sqrt{\frac{k-1}{n-2}})$ for solutions of equation~\eqref{eqn:ode-star} starting from $\sqrt{\frac{k-1}{n-2}} - \epsilon_i$, where $\epsilon_i \downarrow 0$.
Here, $\phi$ can be understood as the $C^{\infty}_{\text{loc}}$ subsequential limit of solutions $\phi^{(M)}$ to equation~\eqref{lemma:analysis-of-limit-equation} with data $\{ \phi^{(M)}(-1) = 0, \phi^{(M)}{}'(-1) = M \}$ as $M \to \infty$, by arguing as in Lemma~\ref{lemma:slope-infinity-solution}.
The proof of Proposition~\ref{prop:small-breather}, ruling out solutions with two zeros near $\sqrt{\frac{k-1}{n-2}}$, relies on a blowup argument using the following property of this limiting equation.
\begin{lemma}\label{lemma:analysis-of-limit-equation}
    For $c \geq 2$, let $\phi$ be the maximally extended solution of equation~\eqref{eqn:phi-limit-equation}.
Then, the quantity $\phi - 2 s \phi'$ becomes negative at some $s_* > -1$.
\end{lemma}

\begin{proof}
Suppose, for contradiction, that the result is false, so the function $h := \phi - 2 s \phi'$ is everywhere non-negative.
    At a point $s_*$ where $\phi'(s_*) = 0$, equation~\eqref{eqn:phi-limit-equation} implies $\phi''(s_*) = 0$.
Since $\phi = 0$ is a solution of~\eqref{eqn:phi-limit-equation}, the uniqueness of solutions of ODE would force $\phi \equiv 0$, contradicting the initial data at $s=-1$.
Thus, $\phi' > 0$ for all $s>0$.
The function $h$ satisfies $h' = 2cs \tfrac{(\phi')^2}{1+\phi^2} h - \phi'$, analogous to the Riccati equation~\eqref{eqn:h'-abbrevaited}.
Since $\phi' >0$, this would imply that $h'(s_*) < 0$ at any point where $h(s_*) = 0$, hence $h(s_*+\delta)<0$ for small $\delta>0$, contradicting $h \geq 0$; therefore, $h>0$ for all $s\geq -1$.
Let $S_b \leq \infty$ denote the blowup time of the solution $\phi$ of~\eqref{eqn:phi-limit-equation}.
We write $\phi'' = - c \frac{(\phi')^2}{1 + \phi^2} h < 0$ to see that $\phi'$ is strictly decreasing and positive for all $s \in (-1,S_b)$, therefore $\phi$ cannot have finite-time blowup at any $S_b < \infty$.
This implies that $\phi$ is defined for all $s \in [-1,\infty)$ and $\phi \in C^{\infty}(-1,\infty) \cap C^{\frac{1}{2}}([-1,\infty))$.
Moreover, $\hat{\ell}:= \lim_{s \to \infty} \phi'(s)$ is finite, with $\hat{\ell} \geq 0$.
In fact, $\hat{\ell} > 0$ is impossible, else $\phi'' < 0$ would produce some $S_0$ with $\phi' < \frac{3}{2} \hat{\ell}$ for $s \geq S_0$.
Then,
\[
h(s) = \phi(s) - 2s \phi'(s) \leq \tfrac{3}{2} \hat{\ell} (s - S_0) + \phi(S_0) - 2 \hat{\ell} s < \phi(S_0) - \tfrac{1}{2} \hat{\ell} s
\]
which becomes negative for $s > S_0 + 4 \hat{\ell}^{-1} \phi(S_0)$, contradicting $h(s)>0$.
Thus, $\lim_{s \to \infty} \phi'(s) = 0$.
On the other hand, using $\phi' > 0$ and $\frac{\phi}{1 + \phi^2} \leq \frac{1}{2}$ in equation~\eqref{eqn:phi-limit-equation}, we find $\phi'' \geq - \frac{c}{2} (\phi')^2$, so
\[
(\tfrac{1}{\phi'})' \leq \tfrac{c}{2} \implies \tfrac{1}{\phi'(s)} \leq \tfrac{c}{2}(s + B_1) \implies \phi(s) \geq \tfrac{2}{c} \log(s+B_1) + B_0.
\]
In particular, $\phi \to \infty$ is unbounded as $s \to \infty$.

Since $\phi$ is an increasing function which is positive for $s>-1$, the map $\tau(s) := \on{arcsinh} \phi(s)$ is strictly increasing with image $[0, \infty)$.
Then, $\tau'(s) = \frac{\phi'}{\sqrt{1+ \phi^2}} = \frac{\phi'}{\cosh \tau}$.
We define
\begin{equation}\label{eqn:L(tau)-function}
    L(\tau) := \frac{\sqrt{1 + \phi(s)^2}}{h} = \frac{\cosh \tau}{\phi - 2 s \phi'} = \frac{1}{\tanh \tau - 2 s \tau'(s)}, \qquad \text{for } \; s = \phi^{-1}( \sinh \tau).
\end{equation}
Given the properties $h > 0$ and $h(-1) = \infty$, the function $\frac{\sqrt{1+\phi(s)^2}}{h}$ is well-defined on all of $[-1,\infty)$, and $L(0)=0$.
Moreover, $L(\tau)$ satisfies the equation
\begin{equation}\label{eqn:L(tau)-Riccati-equation}
    \cP[L] :=L'(\tau) - L(\tau)^2 + (c-1) \tanh \tau \, L(\tau) - c = 0 , \qquad L(0) = 0.
\end{equation}
To see this, note that $\phi' = \tau'(s) \cosh \tau$ and $\phi'' = \tau'' \cosh \tau + (\tau')^2 \sinh \tau$, so
\[
\tau'' + (\tau')^2 \tanh \tau + c \, (\tau')^2 ( \tanh \tau - 2 s \tau')  = 0.
\]
On the other hand, the definition of $L$ in~\eqref{eqn:L(tau)-function} implies that $2 s \tau' = \tanh \tau - L^{-1}$, and differentiating this relation allows us to eliminate $\tau''$ from the above equation and obtain~\eqref{eqn:L(tau)-Riccati-equation}.

The properties of $\phi$ obtained above, under the assumption $h>0$, make $L(\tau)$ defined and positive for all $\tau \in [0,\infty)$.
Observe that the function $c \tanh \tau$ is a Riccati subsolution for the operator $\cP$, with $\cP [ c \tanh \tau] = - 2 c \tanh^2 \tau < 0$, so $L(\tau) > c \tanh \tau$ for $\tau>0$.
Equation~\eqref{eqn:L(tau)-Riccati-equation} leads to
\begin{align*}
& L'(\tau) = L(\tau) ( L(\tau) - (c-1) \tanh \tau) + c > \tanh \tau \, L(\tau) + c , \\
& \implies L(\tau) > \tfrac{1}{2} c \cosh \tau \bigl( \pi - 2 \on{arccot}( \sinh \tau) \bigr) > c \cosh \tau, \qquad \text{for } \tau \ge 2.
\end{align*}
Using $\sinh^2 \tau + 1 > 2 \sinh \tau$ for $\tau \in \bR \setminus \{ \on{arcsinh} 1 \}$ implies $\cosh \tau > 2 \tanh \tau$, so for any $c, \tau \geq 2$, we have $(c-1) \tanh \tau < \frac{1}{2} c \cosh \tau < \frac{1}{2} L (\tau)$.
Returning to equation~\eqref{eqn:L(tau)-function}, we find $L'(\tau) > \tfrac{1}{2} L(\tau)^2$ for $\tau \geq 2$, and $L(2) \geq c \cosh 2$.
This forces $L(\tau) \to \infty$ as $\tau \uparrow \tau_*$ for a finite time $\tau_* \leq 2(1+\frac{1}{c})$, a contradiction. 
Therefore, $h = \phi - 2 s \phi'$ must cross zero, so it becomes negative at some $s_* > -1$.
\end{proof}
We can now show that there is a gap between $\sqrt{\frac{k-1}{n-2}}$ and the two zeros of solutions to equation~\eqref{eqn:capillary-ODE}.
\begin{proposition}\label{prop:small-breather}
There exists a $\bar{\ve}_n>0$ such that for any $t_1 \geq \sqrt{\frac{k-1}{n-2}} - \bar{\ve}_n$ and $a>0$, the solution of equation~\eqref{eqn:capillary-ODE} with initial data $\{ f(t_1) = 0, f'(t_1) = a \}$ does not reach a second zero.    
\end{proposition}
\begin{proof}
Following the derivation of equation~\eqref{eqn:capillary-ODE}, we note that if $f$ solves this equation with parameter $\lambda$ and $f'(t_1) = a$, then $\frac{1}{a}f$ solves~\eqref{eqn:capillary-ODE} with parameter $\Lambda := \lambda a^2$ and $f'(t_1) = 1$.
From Proposition~\ref{prop:unifying-proposition}, we know that for any $\lambda,a \in [0,\infty]$, the solution of equation~\eqref{eqn:capillary-ODE} with initial data $\{ f(t_1) = 0, f'(t_1) = a \}$ starting from $t_1 = \sqrt{\frac{k-1}{n-2}}$ remains strictly increasing while defined.
Therefore, the continuous dependence theorem for ODE (applicable also for $a = \infty$, in view of Lemmas~\ref{lemma:slope-infinity-solution} and~\ref{lemma:convergence-of-solutions}) implies this property for all $t_1 \in (\sqrt{\frac{k-1}{n-2}} - \ve(n,k,\lambda,a), \sqrt{\frac{k-1}{n-2}})$, where the size of the interval may additionally depend on the parameters $\lambda,a$.
We seek to prove that $\ve(n,k,\lambda,a)$ has a uniform positive lower bound, independent of any $\lambda,a$, i.e., $\ve_{n,k} := \liminf_{\lambda, a  \in [0,\infty]} \ve(n,k,\lambda,a) > 0$.
Applying the above compactness and continuity properties of the solutions, it suffices to prove the existence of such a bound as $\lambda, a \downarrow 0$ or $\lambda,a \uparrow \infty$.

If $\lambda = \Lambda a^2 \downarrow 0$, the discussions of Section 3.2 and Corollary 3.11 in~\cite{FTW-1} show that the solutions $\frac{1}{a} f$ converge uniformly to the solution of $\cL_{n,k} f_0 = 0$ with initial data $\{ f_0(t_1) = 0, f'_0(t_1) = 1 \}$.
Applying Lemma~\ref{lemma:harmonic-barrier} shows that there exists a $\ve_{n,k,0}>0$ such that this function remains positive and does not reach a second zero, if $t_1 > \sqrt{\frac{k-1}{n-2}} - \ve_{n,k,0}$; this addresses the first case.
Next, we apply the rescaling property of~\eqref{eqn:capillary-ODE} again to work with the function $\sqrt{\lambda} f$, which solves~\eqref{eqn:ode-star} (i.e., $\lambda=1$).
By compactness and the above discussion, it remains to prove the uniform positivity of $\ve_{n,k} := \liminf_{a \to \infty} \ve(n,k,a) > 0$ among solutions of $\{ f(t_1) = 0, f'(t_1) = a \}$. 
Applying Lemmas~\ref{lemma:slope-infinity-solution} and~\ref{lemma:convergence-of-solutions} again, we are reduced to proving the result for solutions $f_{t_1}$ of~\eqref{eqn:ode-star} with vertical slope starting from $t_1$, i.e., $\{ f(t_1) = 0, f'(t_1) = \infty\}$.

Suppose, for contradiction, that the result does not hold, so there exists a sequence of $\epsilon_i \downarrow 0$ such that the above solution $f_i := f_{t_{\alpha}- \epsilon_i}$ starting from $t_{\alpha} - \epsilon_i$ reaches a second zero; here and in what follows, we set $t_{\alpha} := \sqrt{\frac{k-1}{n-2}}$, for brevity.
Applying Proposition~\hyperref[prop:unifying-proposition]{\ref{prop:unifying-proposition}$(iii)$}, we have $h_i = f_i - A f'_i > 0$.
We define $t_i(s) := t_{\alpha} + \epsilon_i s$ and let $\phi_i(s) := f_i(t_i(s))$, so $\phi'_i(s) = \epsilon_i f'_i(t_i(s))$ and $\phi''_i(s) = \epsilon_i^2 f''_i(t_i(s))$.
Moreover, $A(t_{\alpha}) = 0, A'(t_{\alpha}) = 2$ expresses $A(t_i(s)) = 2 \epsilon_i s + O (\epsilon_i^2)$, and 
\begin{equation}\label{eqn:hi-term-to-h-infty}
h_i = \phi_i - \epsilon_i^{-1} A \phi'_i = \phi_i - 2 s \phi'_i - O(\epsilon_i) \phi'_i.
\end{equation}
The initial condition for $f_i$ becomes $\phi_i(-1) = 0 , \phi'_i(-1) = \infty$.
The rescaled equation becomes
\[
(1-t^2) \phi''_i + (n-2) (1-t^2) \tfrac{(\phi'_i)^2}{1 + \phi_i^2} h_i = \epsilon_iR_i(s), \qquad R_i(s) := t_i(s) \phi'_i -\epsilon_i\phi_i - (n-2) \epsilon_i \tfrac{h_i}{1 + \phi_i^2},
\]
and we see $R_i$ is uniformly bounded, so $\epsilon_i R_i \downarrow 0$. 
Since $h_i > 0$ for every $i$, the above equation implies 
\[
(1-t^2) \phi''_i \leq - (n-2) (1-t^2) \tfrac{(\phi'_i)^2}{1 + \phi_i^2} h_i + \epsilon_i R_i \leq C \epsilon_i|\phi_i'| + C \epsilon_i^2|\phi_i|
\]
which implies that $\limsup_{\epsilon_i \downarrow 0} \| \phi''_i \|_{C^0(-1+\delta,S)}$ is uniformly bounded by a constant depending only on $\delta,S>0$.
Note that the uniform boundedness on intervals with $s \geq -1+\delta$ follows as in the convergence and compactness arguments of Lemmas~\ref{lemma:slope-infinity-solution} and~\ref{lemma:convergence-of-solutions}.
Therefore, $\phi_i, \phi'_i$ are also uniformly bounded, and we can apply the Arzel\`a-Ascoli theorem and a diagonal argument over intervals $[ -1 + \frac{1}{k}, S_k]$, as in Lemma~\ref{lemma:slope-infinity-solution}, to extract a subsequential limit $\phi_{i_k} \to \phi_{\infty}$ in $C^2$, hence also in $C^{\infty}$, as solutions of elliptic equations with right-hand sides converging uniformly.
The uniform bounds on $\| \phi_i \|_{C^2}$ produce $R_i(s) \to 0$ uniformly as $\epsilon_i \to 0$, hence $\phi$ satisfies equation~\eqref{eqn:phi-limit-equation} with $c=n-2$.

To obtain the limit equation~\eqref{eqn:phi-limit-equation}, we first observe that the factor $1-t^2$ can be cancelled because it is uniformly positive due to $t_i(s) \to t_{\alpha}$ as $\epsilon_i \downarrow 0$, for any $s \in [-1,S]$.
The initial condition $\phi_{\infty}(-1) = 0$ comes from $\phi_i(-1) = 0$; to see that $\phi'_{\infty}(-1) = \infty$, observe that each $f_i(t)$ satisfies the expansion $f_i(t) \sim \sqrt{\frac{t - t_i(-1)}{(n-2) \epsilon_i}}$ and $f'_i(t) \sim \frac{1}{2} \sqrt{\frac{1}{(n-2) \epsilon_i}} \frac{1}{\sqrt{t - t_i(-1)}}$ for $0 < t- (t_{\alpha} - \epsilon_i) \ll \epsilon_i$.
This property comes from the expansions~\eqref{eqn:f(t)f'(t)-expansion} together with $A(t_{\alpha} - \epsilon_i) = - 2 \epsilon_i + O(\epsilon_i^2)$ as above.
For $t = t_i(s)$ with $s \in (-1,-1+\delta]$, this transformation produces $\phi_i(s) \sim \sqrt{\frac{s+1}{n-2}}$ and $\phi'_i(s) \sim \frac{1}{2} \sqrt{\frac{1}{(n-2)(s+1)}}$.
Therefore, the limit solution $\phi_i \xrightarrow{C^{\infty}_{\text{loc}}} \phi_{\infty}$ has vertical initial slope, and can be obtained via the local series expansion $\phi_{\infty}(s) = \sqrt{\frac{s+1}{n-2}} + O(s+1)$ near $s=-1$, as in Lemma~\ref{lemma:convergence-of-solutions}.

By the considerations of Lemmas~\ref{lemma:slope-infinity-solution} and~\ref{lemma:convergence-of-solutions}, the function $\phi_{\infty}$ agrees with the $C^{\infty}_{\text{loc}}$ subsequential limit of solutions $\phi^{(M)}$ to~\eqref{eqn:phi-limit-equation} with initial data $\{ \phi^{(M)}(-1) = 0, \phi^{(M)}{}'(-1) = M \}$ as $M \to \infty$, which is precisely the solution $\phi$ considered in Lemma~\ref{lemma:analysis-of-limit-equation}.
Moreover, the functions $h_i(s)$ from~\eqref{eqn:hi-term-to-h-infty} satisfy $h_i \to h_{\infty} := \phi_{\infty} - 2 s \phi'_{\infty}$, hence $h_{\infty} \geq 0$ is everywhere non-negative due to $h_i >0$.
However, this contradicts Lemma~\ref{lemma:analysis-of-limit-equation}, which shows that $h_{\infty}$ must become negative.
We conclude that such a sequence of $\epsilon_i \downarrow 0$ as above cannot exist, hence $f_{t_1}$ cannot reach a second zero if $t_1 > t_{\alpha} - \bar{\ve}_{n,k}$.
Letting $\bar{\ve}_n := \min_{2 \leq k \leq n-2} \bar{\ve}_{n,k}$ and applying the above arguments completes the proof.
\end{proof}

The other key ingredient of our proof is the following result, which shows that for any given initial slope $a \in (0,\infty]$, solutions starting sufficiently close to $t_1 = 0$ must reach a second root.

\begin{proposition}\label{prop:shoot-infinite-slope}
There exists a $\ve_n > 0$ such that for any $t_1 \in (0,\ve_n)$ and $a \in (0,\infty]$, the solution of equation~\eqref{eqn:capillary-ODE} with initial data $\{ f(t_1) = 0, f'(t_1) = a \}$ reaches a second zero in its domain.
\end{proposition}
\begin{proof}
Applying the reductions in the first half of Proposition~\ref{prop:small-breather}, it suffices to prove this result in the cases $\lambda=0$ and $\lambda=1$, from which a uniform $\ve_{n,k}>0$ may be extracted.
We then let $\ve_n := \min_{2 \leq k \leq n-2} \ve_{n,k} > 0$.
The first case follows from Lemma~\ref{lemma:harmonic-barrier}, so we focus on solutions of equation~\eqref{eqn:ode-star}.
The proof proceeds in two steps.

\smallskip \noindent \textbf{Step 1:}
First, we prove that for a small $\delta_{n,k}>0$, any solution of equation~\eqref{eqn:ode-star} with initial data
\begin{equation}\label{eqn:original-assumption}
t_1 \in (0,\delta), \qquad (f(t_1) , f'(t_1) ) \in (0,\delta) \times (0,\delta)
\end{equation}
reaches a zero $t_2> \sqrt{\frac{k-1}{n-2}}$ with finite derivative.
We start by considering the linear equation $\cL_{n,k} \tilde{f} = 0$.
From Proposition~\ref{prop:unifying-proposition} (see also~\cite{FTW-1}*{Lemma 3.6}), any solution $\tilde{f}$ of $\cL_{n,k} \tilde{f} = 0$ that is not everywhere positive has a unique critical point, and a zero $t_2 < 1$ where $\tilde{f}'(t_2) < 0$, while $\tilde{f}'<0$ on $(t_2,1)$.
Since the hypergeometric function $\tilde{f}(t) = {}_2F_1(\frac{n-1}{2}, - \frac{1}{2} ; \frac{k}{2} ; t^2)$ is the solution corresponding to the initial data $\{ \tilde{f}(0) = 1, \tilde{f}'(0) = 0 \}$, and reaches a zero strictly before $t=1$ with finite derivative, the smooth dependence of solutions on the initial data shows that the same property holds for solutions $\tilde{f}_{t_1,0}$ with data $\{ \tilde{f}(t_1) = 1, \; \tilde{f}'(t_1)= 0 \}$ starting from any $t_1 \in (0, 2 \hat{\delta})$, provided that $\delta = \hat{\delta}_{n,k}$ is sufficiently small.
On the other hand, Lemma~\ref{lemma:harmonic-barrier} implies that solutions $\tilde{f}_{t_1,1}$ with data $\{ \tilde{f}(t_1) = 0, \tilde{f}'(t_1) = 1 \}$ have the same property, for $t_1 \in (0,2 \delta)$.
Finally, we observe that the solution $\tilde{f}_{t_1,c_0,c_1}$ to equation $\cL_{n,k} \tilde{f} = 0$ with initial data $\{ \tilde{f}(t_1) = c_0, \tilde{f}'(t_1) = c_1 \}$ is given by
\[
\tilde{f}_{t_1,c_0,c_1}(t) = c_0 \tilde{f}_{t_1,0} (t) + c_1 \tilde{f}_{t_1,1}(t)
\]
where each $\tilde{f}_{t_1,i}(t)$ reaches a zero at $t_{2,i} > \sqrt{\frac{k-1}{n-2}}$, has a unique critical point, and is strictly decreasing for $t > t_{2,i}$, as observed above.
Consequently, $\tilde{f}_{t_1,c_0,c_1}$ must also reach zero at a point $t_{2,c_0,c_1} \leq \max \{ t_{2,0}, t_{2,1}\}$.
Shrinking the $t_1$-interval, if needed, and shrinking the initial data space to $c_0,c_1 \in (0,2)$ (since the equation $\cL_{n,k} \tilde{f} = 0$ is scale-invariant), we may again apply the smooth dependence theorem to obtain a small $\lambda_{n,k}>0$ such that every solution of equation~\eqref{eqn:capillary-ODE} with parameter $\lambda \in [0,\lambda_{n,k})$ and initial data
\begin{equation}\label{eqn:initial-data-existence}
t_1 \in (0,\delta), \qquad (f(t_1), f'(t_1)) \in (0,2) \times (0,2)
\end{equation}
reaches a zero $t_2 > \sqrt{\frac{k-1}{n-2}}$ with finite derivative.
Finally, we observe that property~\eqref{eqn:original-assumption} implies that the function $f_{\delta} (t):= \frac{1}{\delta} f(t)$ is a solution of equation~\eqref{eqn:capillary-ODE} with constant $\lambda = \delta^2$ and initial data that satisfy~\eqref{eqn:initial-data-existence}.
Consequently, we may redefine $\delta_{n,k} := \min \{ \sqrt{\lambda_{n,k}} , \frac{1}{2} \hat{\delta}_{n,k} \}$ so that all the above arguments are applicable, and the property~\eqref{eqn:original-assumption} holds, as desired.

\smallskip \noindent \textbf{Step 2:}
We claim that for $\ve>0$ sufficiently small and $t_1 \in (0,\ve)$, the solution $f$ of equation~\eqref{eqn:ode-star} with initial data $\{ f(t_1) = 0 , f'(t_1) = a \}$ with \textit{arbitrary} $a \in (0,\infty]$ satisfies the conditions~\eqref{eqn:original-assumption} at a slightly later time $\hat{t}_1$.
Step 1 will then imply that this function reaches a zero $t_2 > \sqrt{\frac{k-1}{n-2}}$ with finite derivative.
This result is true for solutions of equation~\eqref{eqn:capillary-ODE} by the previous rescaling considerations.
To prove the claim, we bound the quantity $f''$ by differential inequalities for~\eqref{eqn:ode-star}, as in Lemma~\ref{lemma:slope-infinity-solution}.
Let $C(\delta) := 2 (1 + 10^2 \delta^{-2})$ for $\delta$ as in~\eqref{eqn:original-assumption} and  take $0 < \ve \ll \frac{\delta^2}{200 \, C(\delta)^2}$ sufficiently small.
For $f \leq \frac{\delta}{10}$ and $t \leq 2C(\delta)\ve$, we bound $A_{n,k}(t) \leq - \tfrac{3}{4} \frac{k-1}{n-2} t^{-1}$ and
\begin{equation}\label{eqn:f''-Bernoulli-bound-v2}
f'' <- \tfrac{k-1}{2} t^{-1} f'(1+ (f')^2) \quad \implies \quad f' \leq \frac{1}{\sqrt{\left(1 + \tfrac{1}{f'(t_1)^2}\right)\left(\frac{t}{t_1}\right)^{k-1}-1}} \leq \frac{1}{\sqrt{(\frac{t}{t_1})^{k-1}-1}} 
\end{equation}
for any solution of~\eqref{eqn:ode-star} starting from $t_1 \in (0,\ve)$ with $f(t_1) = 0$ and $f'(t_1) \in (0,\infty]$ arbitrary.
The second bound comes from comparing this ODE with the solution of the associated Bernoulli equation $y' =- \frac{k-1}{2} t^{-1} y(1+y^2)$, where $y(t) := f'(t)$.
Letting $U(t) := 1 + \frac{1}{y^2}$, we find
\[
U'(t)=- 2 y^{-3} y' \geq (k-1) t^{-1} ( 1 + y^{-2})  = (k-1) t^{-1} U(t).
\]
Integrating this inequality yields $U(t)\geq U(t_1)\bigl(\frac{t}{t_1}\bigr)^{k-1}$, and inverting gives~\eqref{eqn:f''-Bernoulli-bound-v2}. 
Consequently,
\begin{equation}\label{eqn:fBoundFromPolySqrt}
f' \leq \frac{1}{\sqrt{(t/t_1)^{k-1}-1}} \qquad \implies \qquad f(t) \leq \frac{2}{\sqrt{k-1}}\sqrt{t_1(t -t_1)}
\end{equation}
by direct integration.
Using $f \leq \frac{\delta}{100}$ in $\{ f'>0 \} \cap [t_1, 2C(\delta)t_1]$ and $f' \leq \frac{\delta}{10}$ in $\{ f'>0 \} \cap [ C(\delta) t_1, 2 C(\delta) t_1]$, we conclude that at least one of 
\[
\{ f \leq \tfrac{\delta}{2},\ f' = 0 \}
\qquad\text{or}\qquad
\{ f \leq \tfrac{\delta}{2},\ f' \leq \tfrac{\delta}{2} \}
\]
must occur at some point $\hat{t}_1 < 2 C(\delta)t_1 < \delta$, where we used $\ve \ll \frac{\delta^2}{200 C(\delta)^2}$ and $\frac{1}{\sqrt{C(\delta)^{k-1}-1}} \leq \frac{\delta}{10}$.
Therefore, either $f'$ vanishes somewhere in $[t_1,C(\delta)t_1]$, in which case at the first such point $\hat{t}_1$ we have $f(\hat{t}_1)\le \delta/100$ and $f'(\hat{t}_1)=0$, or $f'>0$ on $[t_1,C(\delta)t_1]$ and taking $\hat{t}_1=C(\delta)t_1$ yields $f(\hat{t}_1)\le \delta/100$ and $f'(\hat{t}_1)\le \delta/10$; in either case $\hat{t}_1<2C(\delta)t_1<\delta$ follows from $\ve \ll \frac{\delta^2}{200 C(\delta)^2}$.

In both cases, this shows that $f$ satisfies the assumption~\eqref{eqn:original-assumption} of Step 1, therefore it reaches a zero $t_2 > \sqrt{\frac{k-1}{n-2}}$ with finite derivative as desired.
\end{proof}

\begin{corollary}\label{cor:shooting-matching}
    For any $a \in (0,\infty]$, there exists a $t_1$ such that the solution of equation~\eqref{eqn:ode-star} with initial data $\{ f(t_1) = 0, f'(t_1) = a \}$ reaches a second zero in its domain, at a point $t_2$, where
    \[
    (1 - t_1^2) f'(t_1)^2 = (1 - t_2^2) f'(t_2)^2.
    \]
    If $a = \infty$, this means that $|f'(t)| \to \infty$ as $t \uparrow t_2$.
    The graph of the function $U = \rho f(t)$ defines a capillary minimal cone in $\bR^{n+1}$ with contact angle $\tan \theta = a \sqrt{1 - t_1^2}$.
\end{corollary}
\begin{proof}
For each $t_1 \in ( 0, \sqrt{\frac{k-1}{n-2}} - \bar{\ve}_n)$, let $f_{t_1}$ denote the solution of equation~\eqref{eqn:ode-star} with initial data $\{ f_{t_1}(t_1) = 0 , f'_{t_1}(t_1) = a\}$, starting from the point $t_1$.
    Proposition~\ref{prop:shoot-infinite-slope} shows that if $t_1 \in (0,\ve_n)$, the solution $f_{t_1}$ reaches zero again, at a point $\tau(t_1) > \sqrt{\frac{k-1}{n-2}}$, by Proposition~\hyperref[prop:unifying-proposition]{\ref{prop:unifying-proposition}$(iii)$}.
    On the other hand, taking $t_1 \uparrow \sqrt{\frac{k-1}{n-2}} - \bar{\ve}_n$ makes the solution $f_{t_1}$ remain strictly positive until blowup, by Lemma~\ref{lemma:slope-infinity-solution}.
    Applying Proposition~\ref{prop:shoot-infinite-slope}, we can extract some $t_* \in ( \ve_n , \sqrt{\frac{k-1}{n-2}} - \bar{\ve}_n)$ such that the solution $f_{t_*}$ of $\{ f_{t_*}(t_*) = 0 , f'_{t_*}(t_*) = a \}$ reaches another zero, at a point $\tau(t_*) > \sqrt{\frac{k-1}{n-2}}$, with $f'_{t_*}(\tau(t_*)) = -\infty$.

    If $a = \infty$, then this discussion shows that $|f'(t)| \to \infty$ as $t \to t_*$ or $t \to \tau(t_*)$, so the solution $f_{t_*}$ is the desired profile curve of a free-boundary minimal cone in $\bR^{n+1}_+$. 
    
    If $a < \infty$, we consider any $t_1 \in ( \ve_n, t_*)$ and the corresponding solution $f_{t_1}$ as above, reaching a second zero at $\tau(t_1) > \sqrt{\frac{k-1}{n-2}}$, and form the quantity
    \[
    R[t_1] := (1 - t_1^2) f'_{t_1}(t_1)^2 - (1 - \tau(t_1)^2) f'_{t_1} (\tau (t_1))^2 \,.
    \]
    The smooth dependence on the initial conditions for solutions to ODE with smooth coefficients shows that the map $s \mapsto R[s]$ is smooth whenever finite.
    The definition of $t_*$ in~\eqref{eqn:t-star-definition} shows that $R[t_1] \to - \infty$ as $t_1 \uparrow t_*$, since $f'_{t_*}(\tau(t_*)) = - \infty$.
    On the other hand, Proposition~\ref{prop:shoot-infinite-slope} ensures for any small $\delta>0$, there exists an $\ve(\delta)>0$ such that for any $t_1 \in (0,\ve)$, the function $f_{t_1}$ remains close in $C^2$ to the solution of equation~\eqref{eqn:ode-star} (equivalently, also~\eqref{eqn:legendre-form}, due to $|f|_{C^2} = O (\delta)$) with initial data $\{ f(0) = \delta, f'(0) = 0 \}$.
    The analysis of~\cite{FTW-1}*{Corollary 3.11} shows that for $\delta > 0$ sufficiently small, the solution of equation~\eqref{eqn:ode-star}  reaches zero at a point $t_2(\delta)$, where 
    \[
    (1- t_2(\delta)^2 ) f'(t_2(\delta))^2 = c_{n,k}^{-2} \delta^2 + O(\delta^4) 
    \]
    for a dimensional constant $c_{n,k}$ coming from the hypergeometric function ${}_2F_1(\frac{n-1}{2}, - \frac{1}{2} ; \frac{k}{2} ; t^2)$.
    Therefore, taking $\delta$ sufficiently small and using the smooth dependence of the solutions, we see that at the zero $\tau(t_1) > \sqrt{\frac{k-1}{n-2}}$ of $f_{t_1}$ with $t_1 \in (0 , \ve(\delta))$, it holds that
    \[
    (1 - \tau(t_1)^2 ) f'_{t_1} ( \tau(t_1))^2 < 2 \, c_{n,k}^{-2} \delta^2.
    \]
    We conclude that, for any such $t_1$,
    \[
    R[t_1] = (1 - t_1^2) f'_{t_1}(t_1)^2 - (1 - \tau(t_1)^2) f'_{t_1}(\tau(t_1))^2 \geq (1 - \delta^2) a^2 - 2 c_{n,k}^{-2} \delta^2
    \]
    which implies that $R[t_1] > 0$ for $\delta$ sufficiently small (in particular, $\delta < \frac{c_{n,k}}{2}a$).
    By the continuity of $R[t_1]$ in $t_1 \in (\ve(\delta), t_*)$, we conclude that there exists a $t_1(a) \in (\ve(\delta), t_*)$ such that $R[t_1(a)] = 0$.
    The discussion of Section~\ref{section:topology} (cf.~\cite{FTW-1}*{Proposition 3.2}) shows that the graph of $U = \rho f_{t_1}(t)$ is a capillary minimal cone in $\bR^{n+1}_+$ with contact angle $\tan\theta = a \sqrt{1 - t_1^2}$.
\end{proof}
For fixed $\lambda \in (0,\infty)$, a solution of equation~\eqref{eqn:capillary-ODE} whose graph $U = \rho f(t)$ is a capillary minimal cone can be identified by the data $\{ t_1, a = f'(t_1), t_2 = \tau(t_1) \}$ via $\theta = \arctan \sqrt{(1 - t_i^2) f'(t_i)^2}$.
In view of Lemma~\ref{lemma:change-of-variables}, we obtain the symmetries
\begin{equation}\label{eqn:interchange-n-k-symmetries}
\begin{aligned}
    t_2(n,n-k,a) &= \sqrt{1 - t_1(n,k,a)^2}, \qquad & \theta(n,k,a, t_1) &= \theta(n,n-k,a, \sqrt{1-t_1^2}).
\end{aligned}
\end{equation}

\section{Proof of the main theorems}

Using the above results, notably Corollary~\ref{cor:shooting-matching}, we can now prove Theorems~\ref{thm:MinimalSurfaceExistence}, ~\ref{thm:capillaryConeExistence}, and~\ref{thm:carlotto-schulz}.

\begin{proof}[Proof of Theorem~\ref{thm:MinimalSurfaceExistence}]
Using Corollary~\ref{cor:shooting-matching} for $a = \infty$, we obtain a solution $f_{t_*}$ of~\eqref{eqn:capillary-ODE} such that $\tilde{\mathbf{C}}_{n,k,\frac{\pi}{2}} := \text{graph}( \rho f_{t_*})$ is a free-boundary minimal cone in $\bR^{n+1}_+$.
Doubling $\tilde{\mathbf{C}}_{n,k,\frac{\pi}{2}}$ over the plane $\{z=0\}$ produces a complete minimal hypercone $\hat{\mathbf{C}}_{n,k} \subset \bR^{n+1}$.
    The discussion of Section~\ref{section:topology} shows that the link $\mathbf{S}_{n,k} := \hat{\mathbf{C}}_{n,k} \cap \bS^n$ is a minimal hypersurface diffeomorphic to $\bS^{n-k-1} \times \bS^{k-1} \times \bS^1$, for given $n \geq 4$ and $2 \leq k \leq n-2$.
Setting $p=n-k-1 \geq 1$ and $q = k-1 \geq 1$ proves Theorem~\ref{thm:MinimalSurfaceExistence}.
\end{proof}

\begin{proof}[Proof of Theorem~\ref{thm:capillaryConeExistence}]
    The proof of this result is analogous to Corollary~\ref{cor:shooting-matching}.
    Since the case $\theta = \frac{\pi}{2}$ is already settled, we consider $\theta \in (0, \frac{\pi}{2})$.
    For each $t_1 \in (0, \sqrt{\frac{k-1}{n-2}} - \bar{\ve}_n)$, let $f_{t_1}$ denote the solution of equation~\eqref{eqn:ode-star} with initial data $\bigl\{ f_{t_1}(t_1) = 0, f'_{t_1}(t_1) = \frac{\tan \theta}{\sqrt{1 - t_1^2}} \bigr\}$.
    As in the previous proof, for $t_1 \in (0, \ve(\theta))$ sufficiently small, the solution $f_{t_1}$ reaches a second zero at $\tau(t_1) > \sqrt{\frac{k-1}{n-2}}$, where
    \[
    R[t_1]:= (1-t_1^2)f'_{t_1}(t_1)^2 - (1 - \tau(t_1)^2) f'_{t_1}(\tau(t_1))^2 > \tan^2 \theta - 2 c_{n,k}^{-2} \delta^2 > 0.
    \]
    On the other hand, $t_1 \uparrow t_*$ makes $R[t_1] \to - \infty$, so there exists a $t_1(\theta) \in (\ve(\theta), \sqrt{\frac{k-1}{n-2}} - \bar{\ve}_n)$ such that $R[t_1(\theta)] = 0$ and $(1-t_1^2) f'_{t_1}(t_1)^2 = (1 - t_2^2) f'_{t_1}(\tau(t_1))^2 = \tan^2 \theta$. 
    Then, $\tilde{\mathbf{C}}_{n,k,\theta} := \text{graph} ( \rho f_{t_1(\theta)})$ is a capillary cone with contact angle $\theta$ and the topology described in Section~\ref{section:topology}, as desired.
\end{proof}
Finally, as discussed in the introduction, our methods give a short proof of the existence of embedded minimal hypertori $\bS^p \times \bS^p \times \bS^1 \hookrightarrow \bS^{2p+2}$ with an interpolation between the free-boundary and homogeneous one-phase solutions by capillary cones that produce every contact angle.
The $O(k) \times O(k)$-invariant one-phase cone in $\bR^{2k}$ was studied by Hong~\cite{hong-singular}.
\begin{proof}[Proof of Theorem~\ref{thm:carlotto-schulz}]
    For every $a>0$, let $f_a$ denote the solution of equation~\eqref{eqn:ode-star} with $n=2k$ and initial data $\{ f(\frac{1}{\sqrt{2}}) = a, f'(\frac{1}{\sqrt{2}}) = 0 \}$.
    For any such solution $f_a$ reaching a zero at $t_2 > \frac{1}{\sqrt{2}}$, the invariance of equation~\eqref{eqn:ode-star} (for $n=2k$) under the involution $t \leftrightsquigarrow \sqrt{1-t^2}$ allows us to reflect $f_a$ about its critical point at $\frac{1}{\sqrt{2}}$ into a smooth solution on an interval $[t_1,t_2]$, with $t_1^2 + t_2^2 = 1$.
    For this, we define $f_a(t) := f_a(\sqrt{1-t^2})$ which produces a smooth function on $[t_1,t_2]$.
    Moreover, the computation of Lemma~\ref{lemma:change-of-variables} shows that $(1-t_1^2) f'_a(t_1)^2 = (1 - t_2^2) f'_a(t_2)^2$, hence the graph of $\rho f_a(t)$ over $[t_1,t_2]$ defines a capillary cone of contact angle $\theta = \arctan ( a \sqrt{1-t_2^2})$. 
    
    For small $a>0$, rescaling $\frac{1}{a} f_a$ produces a solution of equation~\eqref{eqn:capillary-ODE} with parameter $\lambda = a^2 \downarrow 0$.
    As in the discussion of Proposition~\ref{prop:small-breather} and Corollary~\ref{cor:shooting-matching} (after~\cite{FTW-1}*{\S 3}), the rescaled functions converge uniformly in $C^{\infty}_{\text{loc}}$ to the solution $f_0$ of the linear problem
    \[
    \cL_{2k,k} f := (1-t^2) f''_0 + (f_0 - tf'_0) + 2(k-1) \bigl( f_0 - ( t - \tfrac{1}{2t}) f'_0 \bigr) = 0, \qquad \bigl\{ f_0(\tfrac{1}{\sqrt{2}}) = 1, f'_0(\tfrac{1}{\sqrt{2}}) = 0 \bigr\}.
    \]
    We prove that $f_0$ reaches a zero at $t<1$, hence all functions $f_a$ will have such a zero for small $a>0$.
    Let $F(t) := {}_2 F_1 ( \frac{2k-1}{2}, - \frac{1}{2} ; k ; t^2)$ be the solution of $\cL_{2k,k} f_0 = 0$ with initial data $\{ f_0(0) = 1, f'_0(0) = 0 \}$ (see also~\cite{FTW-stability-one-phase}*{\S 2}).
    The involution symmetry discussed above shows that the function $F(t) + F(\sqrt{1-t^2}) \in \ker \cL_{2k,k}$ as well, and has a critical point at $\frac{1}{\sqrt{2}}$, by symmetry; consequently,
    \[
    f_0(t) = \tfrac{1}{c} ( F(t) + F(\sqrt{1-t^2})), \qquad \text{for some } \; c > 0.
    \]
    Here, we used the fact that $F(\frac{1}{\sqrt{2}}) > 0$ by an established analysis of the generalized equation~\eqref{eqn:capillary-ODE} (see, for example,~\cite{FTW-1}*{Lemma 3.6}). 
    The hypergeometric function ${}_2F_1(a,b;c;s)$ has a pole $\sim (1-s)^{c-a-b}$, cf.~\cites{FTW-1, FTW-stability-one-phase}.
    In our situation, this implies that $F(t) \sim - (1-t)^{- \frac{n-4}{4}}$ for $n >4$ and $F(t) \sim \log (1-t)$ for $n=4$.
    In either case, we conclude that $f_0(t) \to - \infty$ as $t \uparrow 1$, hence $f_0(t)$ becomes zero in its domain of definition, and so does each $f_a$ for $a \in (0,\ve_0)$.

    On the other hand, if $a > \frac{2}{\sqrt{2k}}$, then Lemma~\ref{lemma:bound-at-sqrt} shows that the solution $f_a$ has finite-time blowup before reaching zero.
    Therefore, we have some 
    \[
    a^*_k := \sup \{ a > 0 : f_a \; \text{ reaches a zero} \} < \tfrac{2}{\sqrt{k}} < \infty.
    \]
    The same argument by continuity used in Corollary~\ref{cor:shooting-matching} shows that the solution $f_{a^*_k}$ reaches zero at a point $t_{a^*_k}$ where $f_{a^*_k}(t_{a^*_k}) = 0$ and $f'_{a^*_k}(t_{a^*_k}) = - \infty$.
    Consequently, each $\bar{\mathbf{C}}_{k,a} := \text{graph} ( \rho f_a(t))$ for $a \leq a^*_k$ is a capillary minimal cone.
    The above discussion shows that the rescaled cones $\frac{1}{a} \bar{\mathbf{C}}_{k,a}$ converge to $\text{graph} ( \rho f_0(t))$, as claimed; in particular, their contact angles $\theta(a)$ satisfy $\theta(a) \asymp a$ for small $a$.
    The smooth dependence theorem for solutions of ODE in terms of their initial data shows that the association $a \mapsto f_a$ is $C^{\infty}_{\text{loc}}(0,a^*_k)$, so the map $a \mapsto \bar{\mathbf{C}}_{k,a}$ produces a smooth family of cones.
    Since $\theta(a^*_k) = \frac{\pi}{2}$, we conclude that the map $a \mapsto \theta(a)$ is surjective on $(0,\frac{\pi}{2}]$, hence the cones $\bar{\mathbf{C}}_{k,a}$ produce all capillary angles.
\end{proof}

Carlotto-Schulz~\cite{carlotto-schulz} additionally proved that their methods extended to the construction of a minimal torus $\bT^4 = \bS^1 \times \bS^1 \times \bS^1 \times \bS^1$ in the $5$-sphere (more generally, a minimally embedded $\bS^{n-1} \times \bS^{n-1} \times \bS^{n-1} \times \bS^1$ in the round sphere $\bS^{3n-1}$) conjectured by Hsiang-Hsiang~\cite{hsiang-hsiang}. 
It would be interesting to produce more general examples of minimal surfaces with four spherical factors by applying the techniques developed in the present paper.

\nocite{*}
\bibliography{ref}

\end{document}